\documentclass{amsart}
 \usepackage{graphicx}
 \usepackage{amssymb}
 \usepackage{amsmath}
 \usepackage{amsthm}
 \usepackage{mathtools}
 \usepackage{tikz}
 \usepackage{psfrag}
 \usepackage{xcolor}

\newtheorem{thm}{Theorem}[section]
\newtheorem{lem}[thm]{Lemma}
\newtheorem{pro}[thm]{Proposition}

\newtheorem{defn}[thm]{Definition}

\theoremstyle{definition}

\newcommand{\N}{\mathbb N}

\newcommand{\R}{\mathbb R}

\renewcommand{\u}{\mathbf u}
\renewcommand{\l}{\lambda}

\newcommand{\F}{\mathbf F}

\newcommand{\U}{\mathcal U}

\newcommand{\e}{\epsilon}

\newcommand{\0}{\mathbf 0}

\newcommand{\tr}{\operatorname{tr}}
\newcommand{\rank}{\operatorname{rank}}

\newcommand{\supp}{\operatorname{supp}}

\newcommand{\dist}{\operatorname{dist}}

\newcommand{\cof}{\operatorname{cof}}

\newcommand{\wcon}{\rightharpoonup}
\newcommand{\imply}{\Longrightarrow}

\makeatletter
\numberwithin{equation}{section}
\makeatother

\begin{document}

\author{Seonghak Kim}
\address{Department of Mathematics\\ College of Natural Sciences\\
Kyungpook National University\\
Daegu 41566, Republic of Korea}
\email{shkim17@knu.ac.kr}

 \author{Baisheng Yan}
 \address{Department of Mathematics\\ Michigan State University\\ East Lansing, MI 48824, USA}
   \email{yanb@msu.edu}

\title[Integral convexity, variational solutions  and  semigroups]{On integral convexity, variational solutions\\  and  nonlinear semigroups}

\subjclass[2010]{34G25, 35K87,  47H05, 47H20, 49J45}
\keywords{Integral convexity, variational solutions,  semigroups,  monotonicity conditions}

 \begin{abstract}
 In this paper we provide a different approach for  existence of the variational solutions of the gradient flows associated to  functionals on Sobolev spaces studied in \cite{BDDMS20}. The crucial condition is  the convexity of the functional under which we show that  the variational solutions coincide with the solutions generated by the nonlinear semigroup associated to the functional. For integral functionals of the form  $\mathbf F(u)=\int_\Omega f(x,Du(x)) dx,$ where $f(x,\xi)$ is  $C^1$ in $\xi$,  we also make some remarks on  the connections between convexity of $\F$  (called the integral convexity of $f$) and  certain  monotonicity conditions of the gradient map $D_\xi f.$ In particular, we provide an example to show that even for functions of the simple form $f=f(\xi)$, the usual quasimonotonicity of $D_\xi f$  is not  sufficient  for the integral convexity of $f.$
    \end{abstract}

\maketitle

\section{Introduction}

Let $n,N\ge 1$ be integers and denote by $\R^{N\times n}$  the usual Euclidean space of $N\times n$ real matrices with   inner product
$
\xi:\eta=\tr(\xi^T\eta)
$ and norm $|\xi|=(\xi:\xi)^{1/2}$ for all $\xi,\eta\in \R^{N\times n}.$

We  consider the  variational integral functional
\begin{equation}\label{fun-1}
\mathbf F(u)=\int_\Omega f(x,u(x),Du(x))\,dx,
 \end{equation}
where  $\Omega$ is a bounded domain in $\R^n$, $u=(u^1,\cdots,u^N)\colon \Omega \to\R^N$ is a vector-valued function belonging to some Sobolev space and $f\colon \Omega\times \R^N\times \R^{N\times n}\to  \R$ is a  Carath\'eodory function; namely, $f(x,z,\xi)$ is continuous in $(z,\xi)\in \R^N\times \R^{N\times n}$ for almost every  $x\in\Omega$ and measurable in $x\in\Omega$ for all $(z,\xi)\in \R^N\times \R^{N\times n}.$  Here $
 Du(x)=(\frac{\partial u^i}{\partial {x_j}})\in \R^{N\times n}$ denotes the  Jacobian matrix of weak derivatives of $u$ at $x$.

 Let $1\le p\le \infty$ be given. Under certain growth assumptions on function $f(x,z,\xi)$ (not to be specified here), the functional $\F$ will take the extended values in $\bar\R=\R\cup\{+\infty\}$ on the Sobolev space $W^{1,p}(\Omega;\R^N).$   Given a function $u_*\in L^2(\Omega;\R^N)\cap W^{1,p}(\Omega;\R^N)$  satisfying  $\F(u_*)<\infty,$ we  study  the functional $\F(u)$ on the Dirichlet class $u_*+W^{1,p}_0(\Omega;\R^N)$  with boundary datum $u_*.$

 Note that the function $
 \tilde f(x, z, \xi)=f(x,u_*(x)+ z, Du_*(x)+ \xi)
$ is  Carath\'eodory  on $\Omega\times \R^N\times \R^{N\times n}$ and that $\F(u)=\tilde \F(\tilde u),$ where
 $u=u_*+\tilde u$   and  $\tilde \F(\tilde u)$ is the variational integral  functional on $\tilde u\in W^{1,p}_0(\Omega;\R^N)$ defined by function $\tilde  f;$ furthermore, $\tilde \F$ satisfies that $\tilde \F(0)<\infty$ and takes  the extended values in $\bar \R$ on the space $W_0^{1,p}(\Omega;\R^N).$ Therefore, we  may assume the boundary datum $u_*=0$ and study the functional $\F(u)$ on  $W_0^{1,p}(\Omega;\R^N)$ for a general Carath\'eodory function $f(x,z,\xi).$

 In fact, we will start with an abstract functional
\begin{equation}\label{abs-F}
 \F\colon W^{1,p}_0(\Omega;\R^N)\to \bar\R\;\; \mbox{with}\;\; \F(0)<\infty,
 \end{equation}
  instead of the specific integral form (\ref{fun-1}).
   We are interested in  certain  evolution solutions  associated to such a functional $\F$  on the space $L^2(\Omega;\R^N)\cap W_0^{1,p}(\Omega;\R^N).$

 First, we recall the following definition given in \cite[Definition 2.1]{BDDMS20}.

  \begin{defn} \label{v-sol} Let $T>0$, $\Omega_T=\Omega\times (0,T)$ and $u_0\in L^2(\Omega;\R^N).$ A measurable function $u\colon \Omega_T\to\R^N$ in the class
  \[
  C^0([0,T];L^2(\Omega;\R^N))\cap L^p(0,T; W_0^{1,p}(\Omega;\R^N))
  \]
   is called  a  {\bf variational solution} {\rm (of  the gradient flow)} associated to the  functional $\F$ as in (\ref{abs-F}) with initial datum $u(0)=u_0$ provided that the variational inequality
\begin{equation}\label{var-sol}
 \begin{split} \int_0^\tau \F(u)\,dt \le &\int_0^\tau \F(v)\, dt + \iint_{\Omega_\tau}  \partial_t v \cdot (v-u)\,dxdt \\
 &+ \frac12 \|v(\cdot,0)-u_0\|_{L^2(\Omega)}^2 - \frac12 \|v(\cdot,\tau)-u(\cdot,\tau)\|_{L^2(\Omega)}^2\end{split}
 \end{equation}
 holds for all $\tau\in (0,T]$ and  $ v\in  L^p(0,T; W_0^{1,p}(\Omega;\R^N))$ with $\partial_t v\in L^2(\Omega_T;\R^N)$ and $v(0) \in L^2(\Omega;\R^N).$
  \end{defn}

 We now review some known results proved in \cite{BDDMS20}. Let $1<p<\infty.$ It is proved in {\cite[Theorem 6.1]{BDDMS20} that if    $\F$ is coercive and sequentially weakly  lower semicontinuous  on $W_0^{1,p}(\Omega;\R^N)$ and satisfies  that for  every initial datum $u_0\in L^2(\Omega;\R^N)\cap W_0^{1,p}(\Omega;\R^N)$, a  variational solution of the gradient flow associated to the functional $\F$   exists  in the sense of Definition \ref{v-sol}, then $\F$ must be convex on $L^2(\Omega;\R^N)\cap W_0^{1,p}(\Omega;\R^N).$
 Conversely, if $\F$ is  sequentially weakly  lower semicontinuous on $W_0^{1,p}(\Omega;\R^N)$ and  convex on $L^2(\Omega;\R^N)\cap W_0^{1,p}(\Omega;\R^N)$ and satisfies the following conditions:
\begin{eqnarray}
 & \F(u)<\infty \quad   \forall\, u\in C_0^\infty(\Omega;\R^N),\label{ass-0}\\
 &\F(u)\ge \nu \|Du\|_{L^p(\Omega)}^p -L\quad \forall\, u\in W^{1,p}_0(\Omega;\R^N),
\label{strong-coer}
\end{eqnarray}
 where $\nu$ and $L$ are some positive constants,  then it is proved  in  \cite[Theorem 7.3]{BDDMS20}  that     for every initial datum $u_0\in L^2(\Omega;\R^N)$, a variational solution of  the  gradient flow associated to the functional $\F$ exists.  Moreover, the uniqueness of variational solutions is also proved in \cite[Theorem 7.4]{BDDMS20} if $\F$ is coercive and  sequentially weakly  lower semicontinuous on $W_0^{1,p}(\Omega;\R^N)$ and  convex on $L^2(\Omega;\R^N)\cap W_0^{1,p}(\Omega;\R^N)$ and satisfies (\ref{ass-0}).

In this paper, we aim to provide a different approach for  the study of  variational solutions under the crucial convexity assumption of the functional $\F$ on $L^2(\Omega;\R^N)\cap W_0^{1,p}(\Omega;\R^N)$, based on  the {\em nonlinear semigroup} theory (see, e.g., \cite{AMRT, Bar,Br,Ev-b}). We also make some remarks about  the convexity of functional $\F$ on $W_0^{1,p}(\Omega;\R^N)$ when $\F$ is in some special integral forms of (\ref{fun-1}).

For this purpose, we introduce some  sets associated to the functional $\F$ as given in (\ref{abs-F}).  First, let
 \[
  D(\F) =\{u\in L^2(\Omega;\R^N)\cap W_0^{1,p}(\Omega;\R^N): \F(u)<\infty\};
 \]
then, by (\ref{abs-F}),    $D(\F)\ne\emptyset$ as $0\in D(\F).$  Now, given  $u\in D(\F),$ define
 \[
\mathcal F(u)  =\{v\in L^2(\Omega;\R^N): \F(w)\ge \F(u)+(v,w-u)_{L^2} \;\;\forall\, w\in D(\F)\}
 \]
 to be the {\em $L^2$-subdifferential set} of $\F$ at $u$ and
  \[
\mathcal G(u)  =\left \{v\in L^2(\Omega;\R^N):  \lim_{\tau\to 0} \frac{ \F(u+\tau w)-\F(u)}{\tau} =(v,w)_{L^2} \;\; \forall\,  w\in D(\F)\right \}
 \]
 to be  the {\em $L^2$-gradient set} of $\F$ at $u.$ Finally define
\[
 D(\mathcal F) =\{u\in D(\F): \mathcal F(u)\ne \emptyset\}\;\;\mbox{and}\;\;  D(\mathcal G) =\{u\in D(\F): \mathcal G(u)\ne \emptyset\}.
 \]
  All these sets are considered as  subsets  of the Hilbert space $H=L^2(\Omega;\R^N)$, where the closure of a set $S$ in $H$ will be denoted by $\bar S.$

We remark that both the evolution inclusions
\[
u'(t)\in - \mathcal F (u(t)) \quad \mbox{and}\quad u'(t)\in -  \mathcal G (u(t))
\]
are usually called  an {\em  $L^2$-gradient flow} associated to the  functional $\F.$  In general,  a gradient flow associated to a given functional on a space depends heavily on the structure of the space and notion of the gradient used; we refer to \cite{AGS, San} for more expositions on the general gradient flows. In this paper, we shall focus only on  the gradient flow defined by the subdifferential set $\mathcal F(u).$ We refer to \cite{KY1,KY2, Ya1} for studies on weak solutions to the $L^2$-gradient flow $u'(t)\in -  \mathcal G (u(t))$  for certain nonconvex functionals using the convex integration method.

Our  key  observation   is  the following theorem that will be derived from the nonlinear semigroup theory.

 \begin{thm}\label{main-1} Let $1<p<\infty.$ Assume that $\F$ is coercive and lower semicontinuous on $W_0^{1,p}(\Omega;\R^N)$ and convex on $L^2(\Omega;\R^N)\cap W_0^{1,p}(\Omega;\R^N).$
 Then for each $u_0\in \overline{D(\F)},$  there exists a unique function $u\in C^0([0,\infty);L^2(\Omega;\R^N))$ with
 \[
 u'\in L^{\infty}(\delta,\infty;L^2(\Omega;\R^N))\quad\forall\,\delta>0
 \]
 and $u(t)\in D(\mathcal{F})$ for all $t>0$
that solves the Cauchy problem of gradient flow:
\begin{equation}\label{CP-0}
 \begin{cases}  u'(t)\in -\mathcal F(u(t)), \quad \mbox{a.e.\;$t>0$},\\
 u(0)=u_0.
 \end{cases}
\end{equation}

Moreover, the  solution  $u$ satisfies the following properties:
\begin{itemize}
\item[(i)] $\F(u)\in L^1(0,T)$ for all $T>0.$
\item[(ii)] $\F(u):(0,\infty)\to\R$ is nonincreasing.
\item[(iii)] $\|u'\|_{L^\infty(\delta,\infty;L^2(\Omega;\R^N))}\le \|A^0(u(\delta))\|_{L^2(\Omega)}$ for all $\delta>0,$ where,  for each $w\in D(\mathcal{F})$,  $A^0(w)\in \mathcal{F}(w)$ is  the unique minimizer of $\|z\|_{L^2(\Omega)}$ over the closed and convex subset $\mathcal{F}(w)$ of $L^2(\Omega;\R^N)$.
\item[(iv)] $-u'(t)=A^0(u(t))$ for a.e. $t>0.$
\item[(v)] If $u_0\in D(\F),$ then $u'\in L^2(0,T;L^2(\Omega;\R^N))$ and $\F(u)\in W^{1,1}(0,T)$ for all $T>0.$
\item[(vi)] If $u_0\in D(\mathcal{F}),$ then $u'\in L^\infty(0,\infty;L^2(\Omega;\R^N)).$
\end{itemize}
  \end{thm}

We denote  the unique  function $u$ in the theorem by $u(t)=S(t)u_0.$ It is shown  that $\{S(t)\}_{t\ge 0}$ is in fact a $C^0$-semigroup of contractions  on $\overline{D(\F)}.$  Theorem \ref{main-1} and further properties of the semigroup solution $u(t)=S(t)u_0$ will be established   through the general nonlinear semigroup theory of which  a complete and self-contained review  is given  in Section \ref{s-2} for the convenience of the reader.

 From Theorem \ref{main-1} and the properties of the semigroup $\{S(t)\}_{t\ge 0}$, we establish the following main existence and asymptotic convergence results.

\begin{thm}\label{main-2}  Let $1<p<\infty.$  Let $\F$ satisfy   (\ref{strong-coer}) and  be  lower semicontinuous on $W_0^{1,p}(\Omega;\R^N)$ and convex on $L^2(\Omega;\R^N)\cap W_0^{1,p}(\Omega;\R^N).$ Let $u_0\in \overline{D(\F)}.$
Then the function $u(t)=S(t)u_0$ is a unique variational solution  of the gradient flow associated to  the functional $\F$ with initial datum $u_0$ in the sense of Definition \ref{v-sol} for all $T>0.$
Furthermore,  the following statements are true:
\begin{itemize}
\item[(i)] $u\in L^\infty(\delta,\infty; W^{1,p}_0(\Omega;\R^N))$ for all $\delta>0.$

\item[(ii)]   If  $u_0\in D(\F),$ then $u\in L^\infty(0,\infty; W^{1,p}_0(\Omega;\R^N))$.

\item[(iii)]   If $\F$  satisfies  (\ref{ass-0}), then  $\overline{D(\F)}=L^2(\Omega;\R^N)$.

\item[(iv)]  If $p\ge \frac{2n}{n+2},$ then     $u(t) \rightharpoonup u^*$ in $L^2(\Omega;\R^N)$ as $t\to\infty$
for some minimizer $u^*$ of $\F$ on $W_0^{1,p}(\Omega;\R^N).$ Moreover, if  $p> \frac{2n}{n+2}$, then  $
u(t)\to  u^*$   in $L^2(\Omega;\R^N)$ and $\F(u(t)) \to  \F(u^*)$ as $t\to\infty.$
\end{itemize}
\end{thm}

This theorem will be proved in Section \ref{s-3} (see Proposition \ref{pf-main-2}). Note that our assumption is slightly weaker than that of \cite{BDDMS20} because we  assume that $\F$ is only  lower semicontinuous on $W_0^{1,p}(\Omega;\R^N).$  We also establish  the uniqueness of a variational solution under the assumption of the theorem.  Part (ii) of the theorem shows  that the result of \cite[Theorems 7.1]{BDDMS20} in fact holds  for all $u_0\in D(\F),$ not just for $u_0=u_*=0;$ part (iii)  recovers the result of  \cite[Theorems 7.3]{BDDMS20}; while part (iv)  follows from a  general theorem on the asymptotic convergence for semigroups (see Theorem \ref{asym-thm}).  Moreover, with part (iv) and Theorem \ref{main-1}(ii), we can see that  if $p> \frac{2n}{n+2},$ then  the energy  $\F(u)$ stabilizes to the minimum level $\F(u^*)$ along the solution $u=S(t)u_0$  as $t\to\infty;$ in fact, the energy function  $\F(u(t))$ is also convex (thus continuous) in $t\in (0,\infty)$ (see \cite[Theorem 3.2]{Br}).

We remark that the semigroup method adopted in this paper avoids  the {\em elliptic regularization} used in the proof of  \cite[Theorem 7.1]{BDDMS20} and the {\em time-discretization implicit Euler  scheme} used in that of  \cite[Theorem 7.3]{BDDMS20}.
However, in both the semigroup method of this paper  and the variational method of \cite{BDDMS20},  the   convexity of  the functional $\F$ on $L^2(\Omega;\R^N)\cap W_0^{1,p}(\Omega;\R^N)$ plays a crucial role.

Regarding  functionals  $\F$ of the variational integral form (\ref{fun-1}), the convexity of $\F$ has been discussed at length in \cite[Sections 4 and 5]{BDDMS20} under certain specific assumptions on  $f(x, z,\xi).$  To emphasize such convexity, as in the case of Morrey's {\em quasiconvexity} with respect to the weak lower semicontinuity  of integral functionals on Sobolev spaces (see, e.g.,  \cite{AF,Ba, D,  Mo}),  we  make the following definition.

 \begin{defn} Let $1\le p\le \infty.$
We say that  a Carath\'eodory function $f\colon \Omega\times \R^N\times \R^{N\times n}\to  \R$  is  {\bf $W^{1,p}_0$-integral convex} if the integral  functional $\F$ defined by (\ref{fun-1})  is convex on $W^{1,p}_0(\Omega;\R^N).$ When $p=\infty$, we simply say that  $f$ is {\bf integral convex.}
 \end{defn}

Observe that if $f(x,z,\xi)$ is $C^1$ in $(z,\xi)$ and satisfies certain growth conditions,  then $f$ is $W_0^{1,p}$-integral convex if and only if the function $h(t)=\frac{d}{dt} \F(u+t\phi)$ is nondecreasing   in $t\in\R$ for all $u,\phi\in W^{1,p}_0(\Omega;\R^N);$ this gives rise to  a  variational inequality ({\em monotonicity}) condition on the gradient map  $(D_z f, D_\xi f).$

For example, for  functions  $f=f(x,\xi)$, where $f$ is $C^1$ in $\xi$,  we have the following condition pertaining to the gradient map  $D_\xi f(x,\xi)\colon \Omega\times  \R^{N\times n}\to  \R^{N\times n}.$

\begin{defn} We say that a Carath\'eodory  map $A(x,\xi) \colon \Omega\times  \R^{N\times n}\to  \R^{N\times n}$ is  {\bf $W^{1,p}_0$-integral monotone} if the  inequality
  \begin{equation}\label{int-mono}
\int_\Omega [ A (x,D\psi+ D\phi)-A(x,D\psi)] :D\phi \,dx \ge 0
\end{equation}
holds for all $\psi,\, \phi\in W^{1,p}_0(\Omega;\R^N).$ When $p=\infty$, we simply say that  $A$ is {\bf integral monotone.}
 \end{defn}

We shall study  $W^{1,p}_0$-integral monotonicity further in Section \ref{s-4}, but here we discuss some consequences of this condition  for the special maps   $A=A(\xi)\colon  \R^{N\times n}\to  \R^{N\times n}.$  For example, in (\ref{int-mono}), let  $\phi,\psi \in C_0^\infty(\Omega;\R^N)$  with $D\psi=\xi$ on the support of $\phi,$ and we  easily see that the $W^{1,p}_0$-integral monotonicity of $A(\xi)$   implies the usual {\em quasimonotonicity} of $A(\xi)$:
\begin{equation}\label{q-mono-0}
\int_\Omega  A (\xi+ D\phi(x)) :D\phi(x) \,dx  \ge 0 \quad \forall\,  \xi\in \R^{N\times n},\; \phi\in C_0^\infty(\Omega;\R^N).
\end{equation}
 In the special case of $A(\xi)=Df(\xi),$ where $f$ is $C^1$, quasimonotonicity condition (\ref{q-mono-0}) implies that the function $h(t)= \int_\Omega f(\xi+tD\phi(x))dx $ satisfies that $h'(t)\ge 0$ for all $t\ge 0$; hence $h(1)\ge h(0)$, which gives the well-known  Morrey's  quasiconvexity for function $f$; that is,
\begin{equation}\label{qx}
\int_\Omega f(\xi+ D\phi(x))\,dx\ge f(\xi)|\Omega|\quad \forall\,  \xi\in \R^{N\times n},\; \phi \in C_0^\infty(\Omega;\R^N)
\end{equation}

We remark that quasimonotonicity  and quasiconvexity play an  important role in the calculus of variations and partial differential  equations; see,  e.g,  \cite{AF,Ba, BDM13, CZ92,  D, DMS11, Ev, Fu87, Ha95,   La96, Mo,  Zh86}.
Note that condition (\ref{qx}) (or even the stronger \emph{polyconvexity} of $f$)  does not imply (\ref{q-mono-0}) for $A=Df$ and thus does not imply the integral convexity of $f$; this can be seen by the polyconvex function $f(\xi)=(\det \xi)^2$ ($N=n\ge 2$) that fails to satisfy  (\ref{q-mono-0}) for $A=Df$  (see, e.g., \cite[Theorem 3.7]{Ya2}).
It is also remarked  that the polyconvex functions on $\R^{2\times 2}$ implicitly constructed in \cite{Ya1} cannot be integral convex either,  because their gradient flows have infinitely many Lipschitz  solutions.

On the other hand, integral convexity is in general not sufficient to guarantee polyconvexity  either. For example, we consider an example by Serre \cite{Se}. Let
\[
\tilde{f}(\xi)=(\xi_{11}-\xi_{32}-\xi_{23})^2+(\xi_{12}-\xi_{31}+\xi_{13})^2 +(\xi_{21}-\xi_{31}-\xi_{13})^2 +\xi_{22}^2 +\xi_{33}^2
\]
for all $\xi=(\xi_{ij})\in\R^{3\times 3}.$ We can choose a number $\epsilon>0$ such that
\[
\tilde{f}(a\otimes b)\ge \epsilon|a\otimes b|^2\quad\forall \,  a, b\in\R^3.
\]
Then it can be shown that the quadratic function $f(\xi)=\tilde{f}(\xi)-\epsilon|\xi|^2$  is rank-one convex (thus also quasiconvex) but not polyconvex (see \cite[Theorem 5.25]{D}). It is easily seen that the quasiconvexity of any  quadratic function implies the integral convexity of the function; thus,  this function $f$ is  integral convex.

Consequently,  polyconvexity is neither necessary nor sufficient for  integral convexity in general.
In this paper, we show that  quasimonotonicity  is not a sufficient condition for integral convexity either.
For this purpose,  we  consider the special  function:
\begin{equation}\label{fun-g}
g(\xi)= |\xi|^4 + k (\det \xi)^2 \quad \forall\, \xi\in \R^{2\times 2},
\end{equation}
where $k\in \R$ is a constant; see, e.g., \cite{CZ92, DDGR, Ha95}.
It has been proved in \cite{CZ92}  that for all $k>4$  the function $g$ is not convex and that for all  $0\le k \le 8$ the gradient  map $A(\xi)=Dg(\xi)$ is  quasimonotone;  see also \cite{Ha95}.

Regarding the special function $g$, we have  the following result, which will be restated and  proved in  more details in Proposition \ref{counter}.

\begin{pro}\label{counter-0}  The function $g$ is  integral convex if and only if $-2\le k\le 4,$ in which case $g$ is in fact convex.
Consequently, for all $4<k\le 8$, $Dg$ is quasimonotone, but  $g$ is not integral convex.
\end{pro}

Finally, we remark in passing that, in the class of functions of the simple form $f=f(\xi),$  the quadratic rank-one convex functions and the general convex functions modulo null-Lagrangians, as well as  all the convex combinations of them,  seem to be  the only functions known to be  integral convex;  it would be interesting to see if there are other types of  integral convex functions. We also summarize in Figure \ref{fig1} some relations between various convexities of $f=f(\xi)\in C^1(\R^{N\times n})$ and monotonicities of its gradient map $A=Df$ when $N,n\ge2$ (see \cite{BDDMS20,D}).

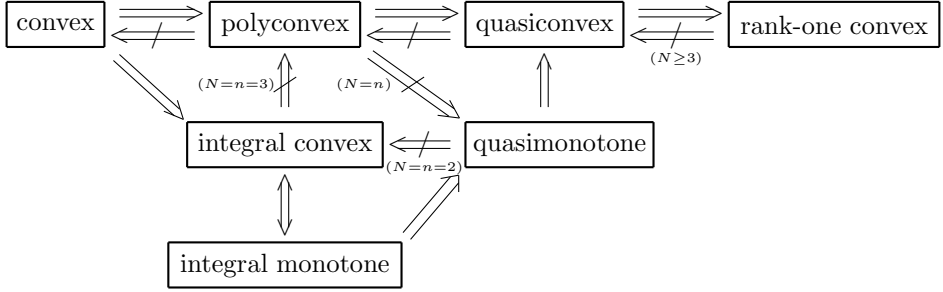
\begin{figure}[ht]
\begin{center}
\begin{tikzpicture}[scale =1]
    \draw[thick] (-0.1,5) -- (1.2,5);
    \draw[thick] (-0.1,4.4) -- (1.2,4.4);
    \draw[thick] (-0.1,4.4) -- (-0.1,5);
    \draw[thick] (1.2,4.4) -- (1.2,5);
    \draw[thick] (0.55,4.7) node[] {convex};
    \draw (1.4,4.9) -- (2.4,4.9);
    \draw (1.4,4.8) -- (2.4,4.8);
    \draw (1.4,4.5) -- (2.4,4.5);
    \draw (1.4,4.6) -- (2.4,4.6);
    \draw (2.4,4.85) node[] {$>$};
    \draw (1.4,4.55) node[] {$<$};
    \draw (1.9,4.55) node[] {$/$};
    \draw (1.4,4.3) -- (2.2,3+63/110);
    \draw (1.32,4.22) -- (2.13,3+133/275);
    \draw (2.17,3+190/275) -- (2.25,3+120/275);
    \draw (1.99,3+150/275) -- (2.25,3+120/275);
    \draw[thick] (2.6,5) -- (4.6,5);
    \draw[thick] (2.6,4.4) -- (4.6,4.4);
    \draw[thick] (2.6,4.4) -- (2.6,5);
    \draw[thick] (4.6,4.4) -- (4.6,5);
    \draw (3.6,4.7) node[] {polyconvex};
    \draw (4.8,4.9) -- (5.8,4.9);
    \draw (4.8,4.8) -- (5.8,4.8);
    \draw (4.8,4.5) -- (5.8,4.5);
    \draw (4.8,4.6) -- (5.8,4.6);
    \draw (5.8,4.85) node[] {$>$};
    \draw (4.8,4.55) node[] {$<$};
    \draw (5.3,4.55) node[] {$/$};
    %
    \draw (4.77,4.35) -- (5.85,3+81/140);
    \draw (4.7,4.25) -- (5.78,3+67/140);
    \draw (5.79,3+190/275) -- (5.92,3+125/275);
    \draw (5.63,3+135/275) -- (5.92,3+125/275);
    \draw (5.15,3.8) -- (5.45,4);
    \draw (4.65,3.9) node[] {$\scriptscriptstyle (N=n)$};
    \draw[thick] (6,5) -- (8.1,5);
    \draw[thick] (6,4.4) -- (8.1,4.4);
    \draw[thick] (6,4.4) -- (6,5);
    \draw[thick] (8.1,4.4) -- (8.1,5);
    \draw (7.05,4.7) node[] {quasiconvex};
    \draw (8.3,4.9) -- (9.3,4.9);
    \draw (8.3,4.8) -- (9.3,4.8);
    \draw (8.3,4.5) -- (9.3,4.5);
    \draw (8.3,4.6) -- (9.3,4.6);
    \draw (9.3,4.85) node[] {$>$};
    \draw (8.3,4.55) node[] {$<$};
    \draw (8.8,4.55) node[] {$/$};
    \draw (8.8,4.25) node[] {$\scriptscriptstyle (N\ge3)$};
    \draw[thick] (9.5,5) -- (12.3,5);
    \draw[thick] (9.5,4.4) -- (12.3,4.4);
    \draw[thick] (9.5,4.4) -- (9.5,5);
    \draw[thick] (12.3,4.4) -- (12.3,5);
    \draw (10.9,4.7) node[] {rank-one convex};
    \draw[thick] (2.3,3.4) -- (4.9,3.4);
    \draw[thick] (2.3,2.8) -- (4.9,2.8);
    \draw[thick] (2.3,2.8) -- (2.3,3.4);
    \draw[thick] (4.9,2.8) -- (4.9,3.4);
    \draw (3.6,3.1) node[] {integral convex};
    \draw (3.65,3.6) -- (3.65,4.2);
    \draw (3.55,3.6) -- (3.55,4.2);
    \draw (3.6,4.2) node[] {$\wedge$};
    \draw (3.45,3.8) -- (3.75,4);
    \draw (2.95,3.9) node[] {$\scriptscriptstyle (N=n=3)$};
    \draw[thick] (2.05,1.8) -- (5.15,1.8);
    \draw[thick] (2.05,1.2) -- (5.15,1.2);
    \draw[thick] (2.05,1.2) -- (2.05,1.8);
    \draw[thick] (5.15,1.2) -- (5.15,1.8);
    \draw (3.6,1.5) node[] {integral monotone};
    \draw (3.65,2) -- (3.65,2.6);
    \draw (3.55,2) -- (3.55,2.6);
    \draw (3.6,2.6) node[] {$\wedge$};
    \draw (3.6,1.99) node[] {$\vee$};
    %
    \draw (5.17,1.93) -- (5.8,2+1141/1700);
    \draw (5.27,1.84) -- (5.91,2+252/425);
    \draw (5.67,2.65) -- (5.91,2+59/85);
    \draw (5.91,2.47) -- (5.91,2+59/85);
    \draw[thick] (6,3.4) -- (8.5,3.4);
    \draw[thick] (6,2.8) -- (8.5,2.8);
    \draw[thick] (6,2.8) -- (6,3.4);
    \draw[thick] (8.5,2.8) -- (8.5,3.4);
    \draw (7.25,3.1) node[] {quasimonotone};
    \draw (7.1,3.6) -- (7.1,4.2);
    \draw (7,3.6) -- (7,4.2);
    \draw (7.05,4.2) node[] {$\wedge$};
    \draw (5.1,3.15) -- (5.8,3.15);
    \draw (5.1,3.05) -- (5.8,3.05);
    \draw (5.1,3.1) node[] {$<$};
    \draw (5.45,3.1) node[] {$/$};
    \draw (5.45,2.80) node[] {$\scriptscriptstyle (N=n=2)$};
    \end{tikzpicture}
\end{center}
\caption{Generalized convexities of $f=f(\xi)\in C^1(\R^{N\times n})$ and monotonicities of $A=Df$ for $N,n\ge2$}
\label{fig1}
\end{figure}

\section{A  Primer on Nonlinear Semigroup Theory}   \label{s-2}

In this section, we provide a self-contained review on the theory of  nonlinear semigroups generated by the subdifferential of  a convex and  lower semicontinuous functional on  Hilbert space.

All the results here can be found in the books \cite{Bar,  Br,  Ev-b}; see also \cite[Chapter 9]{AMRT}.  However,  we include the detailed  proofs for convenience of the reader.

Hereafter, $H$ will denote a real Hilbert space with inner product $(\cdot,\cdot)$ and norm $\|\cdot\|.$  We first recall some basic definitions and properties.

\begin{defn} Let  $I\colon H\to \bar \R =(-\infty,\infty].$
 \begin{itemize}
\item[(i)]   $I$ is called {\bf proper} if  $D(I)=\{u\in H:  I(u)<\infty\}\ne \emptyset.$
\item[(ii)]  $I$ is  called {\bf convex} if $I(\l u +(1-\l)v)\le \l I(u) +(1-\l)I(v)$ for all $u,v\in H$ and $\l\in [0,1].$
\item[(iii)]   $I$ is called  {\bf lower semicontinuous} if
\[
I(u)\le \liminf_{k\to\infty} I(u_k) \quad \mbox{whenever $u_k\to u$ in $H.$}
\]
\item[(iv)] For each $u\in H$,  define the  {\bf subdifferential set} of $I$ at $u$ by
 \[
\partial I(u):=\{v\in H : I(w)\ge I(u)+(v,w-u) \;\;\forall\, w\in H\},
\]
 and define
\[
D(\partial I)=\{u\in H: \partial I(u)\ne \emptyset\}.
\]
\end{itemize}
\end{defn}

The following properties are easily proved  from the definitions; we omit the proof.

\begin{lem}\label{0-lem} Let $I\colon H\to \bar \R $ be  proper and $u,\tilde{u}\in H$.  Then
\begin{enumerate}
\item[(i)] $D(\partial I)\subseteq D(I);$
\item[(ii)] $\partial I(u)$ is  convex and sequentially weakly closed in $H;$

\item[(iii)] \emph{(\mbox{{\bf Monotonicity}})} if $v\in \partial I(u)$ and $\tilde v\in \partial I(\tilde u),$ then
\[
(v-\tilde v, u-\tilde u) \ge 0;
\]
\item[(iv)] $I(u)=\min_H I$ if and only if $0\in \partial I(u).$
\end{enumerate}
\end{lem}

\subsection{Proper, Convex and Lower Semicontinuous  Functionals}
In this subsection, we  assume that  $I\colon H\to \bar \R$ is  a proper, convex and lower semicontinuous functional.

\begin{lem}\label{pro-dense} $\overline{D(\partial I)} =\overline{D(I)}\ne\emptyset.$  Moreover, there exists  $C>0$ such that
\begin{equation}
I(u)\ge  -C -C \|u\| \quad \forall\, u\in H.
\label{coer-I}
\end{equation}
\end{lem}

\begin{proof}
Let  $u\in D(I)$ and $\epsilon >0.$  Let  $h=I(u)-\epsilon.$ Note that  $K=\{[v,s]\in H\times \R : I(v)\le s\}$  is a nonempty, closed and convex set in the Hilbert space $H\times \R.$
Let
$
 [q,k]=\mathbb P_K([u,h])\in K
$
be the projection of $[u,h]\in (H\times\R)\setminus K$ onto the set $K$;
this projection is characterized by $I(q)\le k$ and the variational inequality:
 \begin{equation}\label{C-1}
(p-q,u-q)+(I(p)+\lambda-k)(h-k) \le 0\quad  \forall\, p\in D(I),\;\forall\, \lambda\ge 0.
\end{equation}
In this inequality, letting $\lambda\to \infty$ with $p$ fixed yields  $k\ge h$, and  letting $p=q$ and $ \lambda=0$ gives
$
 (I(q)-k)(h-k)\le 0;
$
thus,
\begin{equation}\label{C-2}
 (I(q)-k)(h-k)= 0.
\end{equation}

We claim that $I(q)>h.$ Otherwise, $I(q)\le h$ would imply $ (I(q)-k)(h-k)\ge (h-k)^2$, and hence, by (\ref{C-2}), we would have  $h=k;$ this in turn by (\ref{C-1}) with $p=u$ would imply $u=q,$   yielding  a contradiction: $I(u)=I(q)\le k=h=I(u)-\epsilon.$

We now have
\[
k\ge I(q)>h=I(u)-\epsilon,
\]
which, by (\ref{C-2}), proves $I(q)=k.$ We thus rewrite   (\ref{C-1})  with $\lambda=0$ as
\[
I(p)\ge I(q) +  \Big(\frac{u-q}{k-h}, p-q\Big) \quad \forall\, p\in D(I).
\]
This shows that $\frac{u-q}{k-h} \in \partial I(q);$  so  $q\in D(\partial I).$ Furthermore,
taking $p=u$ and $\lambda=0$ in (\ref{C-1}) and rearranging, we have
\[
\|u-q\|^2\le\|u-q\|^2 + (I(u)-I(q))^2 \le \epsilon (I(u)-I(q))<\epsilon^2.
\]
This confirms that $\overline{D(\partial I)} =\overline{D(I)}\ne\emptyset.$

Finally, since $D(\partial I)\ne \emptyset,$ we can choose $q_0\in D(\partial I)$ and $p_0\in \partial I(q_0).$ Then
\[
I(u)\ge I(q_0)+(p_0,u-q_0)\ge -C-C\|u\| \quad \forall\, u\in H,
\]
where $C:=\max\{|I(q_0)|+\|p_0\|\|q_0\|,\|p_0\|\}+1.$
\end{proof}

\begin{lem}\label{phi} Given $w\in H$ and $\beta>0,$ let
$
\Phi(u)=\frac{\beta}{2} \|u-w\|^2  + I(u)$ for all $u\in H$. Then
\[
\partial \Phi(u)= \beta (u-w) +\partial I(u) \quad \forall\, u\in D(I).
\]
\end{lem}

\begin{proof}  Let $u\in D(I)$ and $x\in H.$ Set $y=\beta(u-w)+x.$ Then for all $z\in H$,
\[
\Phi(u+z)-\Phi(u)-(y,z)=I(u+z)-I(u)-(x,z)+\frac{\beta}{2}\|z\|^2\quad  \forall\,z\in H.
\]

We claim that the following two statements are equivalent:
\begin{eqnarray}
 I(u+z)-I(u)-(x,z)+\frac{\beta}{2}\|z\|^2\ge 0\quad  \forall\,z\in H. \label{equ-1}
\\
I(u+z)-I(u)-(x,z) \ge 0 \quad \forall\,z\in H.\label{equ-2}
  \end{eqnarray}
Clearly, (\ref{equ-2}) implies (\ref{equ-1}).  Next, assume (\ref{equ-1}). To prove (\ref{equ-2}), without loss of generality, let $z\in H$ be such that $u+z\in D(I)$, and define $h(t)=I(u+tz)-I(u)-t(y,z).$ Then $h$ is convex on $[0,1]$, and $h(0)=0.$ So
 \[
  h(s) \le (1-s)h(0)+ sh(1)=sh(1)\quad \forall\, s\in (0,1).
  \]
Since (\ref{equ-1}) implies $h(s)\ge -\frac{\beta}{2}s^2\|z\|^2$, it follows that $h(1)\ge  -\frac{\beta}{2}s\|z\|^2$ for all $s\in (0,1).$ Letting $s\to 0^+$, we have  $h(1)\ge 0,$ which is (\ref{equ-2}).

Finally, the equivalence of (\ref{equ-1}) and (\ref{equ-2}) implies  that $y\in \partial \Phi(u) \iff x\in \partial I(u),$ which in turn proves that $\partial \Phi(u)= \beta (u-w) +\partial I(u).$
\end{proof}

\begin{lem}\label{pro-onto} For each $w\in H$ and $\l>0$, the inclusion
\begin{equation}\label{inclusion-1}
w\in u+\l \partial I(u)
\end{equation}
has a unique solution $u\in D(\partial I).$
 \end{lem}

\begin{proof}  1. We intend to show by the direct method in the calculus of variations that the functional
\[
\Phi(u)=\frac{1}{2\l} \|u-w\|^2 + I(u)
\]
has a minimizer $u$ on $H,$ which is a solution of inclusion (\ref{inclusion-1}).

First, we show that $\Phi$ is (sequentially) weakly lower semicontinuous on $H$. So, assume $u_k\wcon u$ in $H$ and
\[
\liminf_{k\to \infty} \Phi(u_k)=\lim_{j\to\infty}\Phi(u_{k_j})=:l <\infty.
\]
For each $\epsilon>0$, the set $K_\epsilon=\{w\in H : \Phi(w)\le l+\epsilon\}$ is closed and convex, and is thus  weakly closed by  {Mazur's theorem}. Since all but finitely many of $u_{k_j}$ lie in $K_\epsilon$, by the weak closedness, we have  $u\in K_\epsilon;$ thus
$
\Phi(u)\le l+\epsilon$ for all $\epsilon>0$. Consequently, $\Phi(u)\le l,$ which
  proves the weak lower semicontinuity of $\Phi.$

2. By (\ref{coer-I}), there exists a constant $C_1$ such that
 \begin{equation}\label{coer-J}
\Phi(u)\ge -C_1 +\frac{1}{4\l}\|u\|^2 \quad\forall\,u\in H.
\end{equation}
We choose  a minimizing sequence $\{u_k\}_{k=1}^\infty$ in $H$ for $\Phi$ so that
\[
\lim_{k\to\infty} \Phi(u_k) =\inf_{u\in H} \Phi(u)=:m.
\]
Thanks to (\ref{coer-J}) and the fact that $\Phi$ is proper, we see that $m\in\R$ and that the sequence $\{u_k\}$ is bounded in $H$. So we can choose a subsequence $u_{k_j}\wcon u$ for some $u\in H$. Then the weak lower semicontinuity  of $\Phi$ implies
\[
\Phi(u)\le \liminf_{j\to\infty} \Phi(u_{k_j})=m;
\]
hence,  $u$ is a minimizer of $\Phi$ on $H.$ Thus, by  Lemmas \ref{0-lem}(iv) and \ref{phi},
$
0\in  \partial \Phi(u)=\frac{u-w}{\l} + \partial I(u),
$
from which we have $u\in D(\partial I)$ and $w\in u+\l \partial I(u).$

3. Finally, we show  the uniqueness of  solution $u.$ Assume that $\tilde u\in H$ is any  solution of inclusion (\ref{inclusion-1}). Hence $\frac{w-\tilde u}{\l}\in \partial I(\tilde u)$, and so $\tilde u\in D(\partial I).$ By  monotonicity, we have
\[
\Big(\frac{w-u}{\l}-\frac{w-\tilde u}{\l}, u-\tilde u\Big)\ge 0,
\]
which simplifies to $-\frac{1}{\l}\|u-\tilde u\|^2 \ge 0$ and hence $u=\tilde u$ since  $\l >0.$
\end{proof}

\begin{lem}\label{maximal} The map $\partial I\colon H\to 2^H$ is maximal monotone in the sense that for any $p,q\in H,$
\[
(*) \quad \begin{cases} (p-x,q-y)\ge 0\\
\forall\, x\in D(\partial I),\;\;\forall\,y\in\partial I(x)\end{cases}   \imply \quad p\in D(\partial I),\;\;q\in\partial I(p).
\]
\end{lem}
\begin{proof} Let $p$ and $q$ satisfy $(*).$ By Lemma \ref{pro-onto}, there exists a unique  $x\in D(\partial I)$ such that $p+q\in x+\partial I(x).$ Let $y=p+q-x\in \partial I(x)$ and insert  $x,y$ into $(*)$ to obtain
\[
0\le (p-x,q-y)=(p-x,x-p)=-\|p-x\|^2,
\]
which shows that $p=x$ and hence $q=y.$ So $p\in D(\partial I)$ and $q\in\partial I(p).$
\end{proof}

\begin{defn}  Let  $\l >0.$
\begin{itemize}
\item[(i)] The  {\bf nonlinear resolvent} $J_\l \colon H\to D(\partial I)$ is defined by
\[
J_\l(w)=u,
\]
where $u\in D(\partial I)$ is the unique solution of
$
w\in u+\l \partial I(u).
$
\item[(ii)] The  {\bf Yosida approximation} $A_\l \colon H\to H$ is defined by
\[
A_\l(w)=\frac{w-J_\l(w)}{\l}.
\]

\item[(iii)] The {\bf Moreau envelope}   $I_\l\colon H\to  \R $ is defined by
\[
I_\l(w)=\min_{u\in H} \left\{\frac{1}{2\l} \|u-w\|^2 + I(u)\right \} \quad \forall\, w\in H,
\]
which has a unique minimizer $u=J_\l(w);$   thus
\begin{equation}\label{min-0}
I(w)\ge I_\l(w)=\frac{\l}{2} \|A_\l(w)\|^2 + I(J_\l(w)) \ge I(J_\l(w)) \quad \forall \, w\in H.
\end{equation}
\end{itemize}
\end{defn}

\begin{thm}\label{Yosida}  For each $\l >0$ and $w, \tilde w\in H$, the following statements hold:
\begin{enumerate}
\item[(i)] $\|J_\l(w)-J_\l(\tilde w)\|\le \|w-\tilde w\|.$

\item[(ii)] $\|A_\l (w)-A_\l (\tilde w)\|\le \frac{1}{\l} \|w-\tilde w\|.$

\item[(iii)] $(A_\l(w)-A_\l(\tilde w), w-\tilde w)\ge 0.$

\item[(iv)] $A_\l(w)\in \partial I(J_\l(w)).$

\item[(v)] If $w\in D(\partial I)$, then
$\|A_\l(w)\|\le \|z_0\|$ for all $\l>0,$ where
\[
z_0:=A^0(w)\in \partial I(w)
\]
is the unique minimizer of $\|z\|$ over the closed and convex set $\partial I(w).$

\item[(vi)]
\[
\lim_{\l\to 0^+} J_\l(w)=w \quad \forall\,  w\in \overline{D(\partial I)} =\overline{D(I)}.
\]

\item[(vii)] The  function $I_\l\colon H\to \R $ is convex and differentiable, and $
I_\l'(w)=A_\l(w)$ for all $w\in H;$ thus $I_\l\in C^{1,1}(H;\R).$ Moreover,
\[
\lim_{\l\to 0^+} I_\l(w)=I(w) \quad \forall \, w\in \overline{D(\partial I)} =\overline{D(I)}.
\]
\end{enumerate}
\end{thm}

\begin{proof}
1. Let $u=J_\l(w)$ and $ \tilde u=J_\l(\tilde w).$ Then $u+\l v=w$ and $ \tilde u +\l \tilde v =\tilde w$ for some $v\in \partial I(u)$ and $\tilde v\in \partial I(\tilde u);$ thus,
$
v=A_\l(w)$ and $ \tilde v=A_\l(\tilde w).$ Moreover,
\[
\begin{split}
\|w-\tilde w\|^2 & = \|u-\tilde u +\l(v-\tilde v)\|^2
=\|u-\tilde u\|^2 +2\l (u-\tilde u,v-\tilde v) + \l^2 \|v-\tilde v\|^2 \\
& \ge  \|u-\tilde u\|^2  + \l^2 \|v-\tilde v\|^2
\end{split}
\]
by the monotonicity of $\partial I.$ From this, both (i) and (ii) follow easily.

2.  Following the above, we have
\[
\begin{split}
(A_\l(w)-A_\l(\tilde w), w-\tilde w) & =(v-\tilde v, w-\tilde w)=\frac{1}{\l}(w-\tilde w + \tilde u-u, w-\tilde w)\\
& =\frac{1}{\l} (\|w-\tilde w\|^2 -(w-\tilde w,u-\tilde u))\\
& \ge \frac{1}{\l} (\|w-\tilde w\|^2 -\|w-\tilde w\|\|u-\tilde u\|)\ge 0,
\end{split}
\]
which proves (iii).  Also, it is easily seen that (iv) follows  because $u=J_\l(w)$ is such that $w\in u+\l \partial I(u)$ and thus $ A_\l(w)=\frac{w-u}{\l} \in \partial I(u)=\partial I(J_\l(w)).$

3. We prove (v). Let  $w\in D(\partial I)$, and fix any $\l>0$ and $z\in \partial I(w).$ Since $A_\l(w)\in \partial I(J_\l(w))$, by the monotonicity of $\partial I$, we have
\[
0\le (w-J_\l(w),z-A_\l(w))= \l(A_\l(w), z-A_\l(w)),
\]
which yields that $ \|A_\l(w)\|^2 \le (z,A_\l(w))\le \|A_\l(w)\|\|z\|$; thus
$
\|A_\l(w)\|\le \|z\|.$  So
\[
\sup_{\l>0} \|A_\l(w)\|\le \inf_{z\in \partial I(w)} \|z\|<\infty.
\]
Since $\partial I(w)$ is closed and convex and $\|z\|$ is strictly convex, there exists a unique $z_0=A^0(w)\in \partial I(w)$ such that
\[
\|z_0\|=\min_{z\in \partial I(w)} \|z\|.
\]

4. We prove (vi).  If $w\in D(\partial I)$, then
\[
\|J_\l(w)-w\|=\l \|A_\l(w)\|\le \l \|A^0(w)\|,
\]
and hence $J_\l(w)\to w$ as $\l\to 0^+.$ Now let $w\in \overline{D(\partial I)}\setminus D(\partial I).$ For each $\epsilon>0$, there exists a point $w_\epsilon\in D(\partial I)$ with $\|w_\epsilon-w\|<\e.$ Then
\[
\begin{split}
\|J_\l(w)-w\|& \le \|J_\l(w)-J_\l(w_\e)\|+\|J_\l(w_\e)-w_\e\|+\|w-w_\e\| \\
& \le 2 \|w-w_\e\|+\|J_\l(w_\e)-w_\e\|<2\e +\|J_\l(w_\e)-w_\e\|.
\end{split}
\]
As $w_\e\in D(\partial I)$, we have
\[
\limsup_{\l\to 0^+} \|J_\l(w)-w\|\le 2 \e
\]
for each $\e>0;$ hence $J_\l(w)\to w$ as $\l\to 0^+.$

5. Finally we prove (vii). Since $A_\l(w)\in\partial I(J_\l(w))$ and $J_\l(w)=w-\l A_\l(w)$, we have, for all $v, w \in H$,
\[
\begin{split}
I(J_\l(v))-I(J_\l(w))& \ge (A_\l(w), J_\l(v)-J_\l(w)) \\
& =\l \|A_\l(w)\|^2 +(A_\l(w),v-w) -\l(A_\l(w),A_\l(v)).
\end{split}
\]
Hence, by (\ref{min-0}), for all $v, w \in H$,
\[
\begin{split}
I_\l(v)-I_\l(w) & \ge  \frac{\l}{2} \|A_\l(v)\|^2 +\frac{\l}{2} \|A_\l(w)\|^2 +(A_\l(w),v-w) -\l(A_\l(w),A_\l(v)) \\
& =\frac{\l}{2}  \|A_\l(w)-A_\l(v)\|^2 +(A_\l(w),v-w)
\ge  (A_\l(w),v-w).
\end{split}
\]
Let $v_1,v_2\in H$ and $t\in (0,1);$ set $w=tv_1+(1-t)v_2.$ Then $I_\l(v_k)-I_\l(w)\ge (A_\l(w),v_k-w)$ for $k=1,2;$ thus,  $tI_\l(v_1)+(1-t)I_\l(v_2) -I_\l(w)\ge (A_\l(w),tv_1+(1-t)v_2-w)=0;$  this proves the convexity of $I_\l.$

Interchanging $v$ and $w$ yields
\[
I_\l(v)-I_\l(w)\le (A_\l(v), v-w)
\]
and hence
\[
0\le I_\l(v)-I_\l(w)-(A_\l(w),v-w)\le (A_\l(v)-A_\l(w),v-w)\le \frac{1}{\l} \|v-w\|^2.
\]
Therefore, $I_\l$ is differentiable at $w$ with derivative $I_\l'(w)=A_\l(w)$.

Now assume $w\in \overline{D(\partial I)}.$  By (\ref{min-0}), we have
$
I(J_\l(w))\le I_\l(w) \le I(w),
$
and by (vi),  $J_\l(w)\to w$ as $\l\to 0^+.$ Thus, the lower semicontinuity  of $I$ implies
\[
I(w)\le \liminf_{\l\to 0^+} I(J_\l(w))\le \liminf_{\l\to 0^+} I_\l(w)\le \limsup_{\l\to 0^+} I_\l(w)\le I(w),
\]
which confirms (vii).  \end{proof}

\begin{pro}  Let $A^0\colon D(\partial I)\to H$ be defined above in Theorem \ref{Yosida}(v). Then
 \begin{equation}
\lim_{\l\to 0^+} A_\l (w)=A^0(w) \quad \forall\, w\in D(\partial I).
\label{rk-2.16}\end{equation}
\end{pro}
\begin{proof} Let $w\in D(\partial I)$ and $z_0=A^0(w).$ Note that $\|A_\l(w)\|\le \|z_0\|$ for all $\l>0.$  Let $\l\to 0^+$ be any subsequence such that $A_\l(w)\wcon z $   for some $z\in H$; then  $\|z\|\le \|z_0\|.$  By  Theorem \ref{Yosida}(iv),
\[
I(v)\ge I(J_\l(w))+(A_\l(w),v-J_\l(w))\quad \forall\, v\in H.
\]
So, letting $\l\to 0^+$ along this subsequence, by Theorem \ref{Yosida}(vi) and the lower semicontinuity  of $I$, we have
\[
I(v)\ge I(w) +(z,v-w) \quad \forall\, v\in H,
\]
which shows that $z\in \partial I(w);$ thus  by the uniqueness of $z_0=A^0(w) \in \partial I(w),$ we have $z=z_0.$ This shows that $A_\l(w)\wcon z_0$ as $\l\to 0^+,$ which, due to $\|A_\l(w)\|\le \|z_0\|,$
  implies the strong convergence  $A_\l(w)\to  z_0$ as $\l\to 0^+.$
\end{proof}

\subsection{A Useful  Chain Rule for Differentiation} We  study functions from an interval $\Lambda \subset \R$ into   $H.$ The usual definitions and  basic properties of spaces such as $C^k(\Lambda;H),\, L^p(\Lambda;H)$ and $W^{1,p}(\Lambda;H)$ can be found, e.g.,  in the textbook \cite{Ev-b}.

 We  prove  the following  useful  chain rule for differentiation; see, e.g., \cite[Lemma 4.4]{Bar}.

\begin{lem}\label{chain-rule} Let $I\colon H\to \bar \R$ be proper, convex and lower semicontinuous. Assume $u \in C^{0}([0,T];H)$ with $u'\in L^2(0,T;H)$  and $g\in L^2(0,T;H)$ are such that
$
 g(t)\in \partial I(u(t))$ for  a.e.\;$ t\in (0,T).
$
Then $I(u(t))$ is absolutely continuous on $[0,T]$, and
 \[
 \frac{d}{dt} I(u(t))= (g(t),u'(t)), \quad a.e.\; t\in (0,T).
 \]
 \end{lem}

\begin{proof} Let $I_\l$ be the Moreau envelope. Then $(I_\l(u(t)))'=(I'_\l(u(t)),u'(t))$ for a.e. $t\in(0,T)$, and for all $0\le s<t\le T,$
 \begin{equation}\label{eq-28}
 I_\l(u(t))-I_\l(u(s))=\int_s^t (I'_\l(u(\tau)),u'(\tau)))d\tau=\int_s^t (A_\l(u(\tau)),u'(\tau))\, d\tau.
 \end{equation}
Since $g(\tau)\in \partial I(u(\tau))$ for a.e.\,$\tau\in (0,T),$ we have
 \[
 u(\tau)\in D(\partial I)\;\;\mbox{and}\;\;  \|A_\l(u(\tau))\|  \le \|A^0(u(\tau))\|\le \|g(\tau)\|,  \quad  a.e.\;\tau\in (0,T),\]
 and hence
 \[
  \|(A_\l(u),u')  \|   \le  \|A^0(u)\| \|u'\| \le \|g\|\|u'\|\in L^1(0,T).
 \]
Letting $\l\to 0^+$ in (\ref{eq-28}), by Theorem \ref{Yosida}(vii), Proposition \ref{rk-2.16} and the Lebesgue dominated convergence theorem, we see that the identity
 \begin{equation}\label{eq-29}
 I(u(t))=I(u(s))+\int_s^t (A^0(u(\tau)),u'(\tau))\,d\tau
\end{equation}
holds as long as $0\le s<t\le T$ and $u(t),u(s)\in  D(I).$

Let us check that the identity (\ref{eq-29}) holds for all $0\le s<t\le T.$ First, take any $s\in (0,T)$ such that $g(s)\in \partial I(u(s))$; hence,  $u(s)\in D(\partial I).$ Let $t\in (s,T]$ and  $t_n\in (s,T)$ be such that $t_n\to t$ and $g(t_n)\in \partial I(u(t_n))$; hence,  $u(t_n)\in D(\partial I)$  and $u(t_n)\to u(t).$ In (\ref{eq-29}) with $t=t_n$, by the lower semicontinuity of $I$, we have
\[
I(u(t))\le \liminf_{n\to\infty} I(u(t_n)) =I(u(s))+\int_s^t (A^0(u(\tau)),u'(\tau))\,d\tau<\infty.
\]
This proves that $u(t)\in D(I)$ for all $t\in (s,T].$ Since $s\in (0,T)$ can be arbitrarily small, we have $u(t)\in D(I)$ for all $t\in (0,T]$, and hence   (\ref{eq-29}) holds for all $0<s<t\le T;$ this also shows that
\[
a=\lim_{s\to 0^+} I(u(s))
\]
exists in $\R$. Hence $I(u(0))\le a<\infty$ and so $u(0)\in D(I).$ Therefore,  identity (\ref{eq-29}) holds for all $0\le s<t\le T,$ which proves that the function $I(u(t))$ is absolutely continuous on $[0,T].$

Next, let $t_0\in (0,T)$ be such that  $g(t_0)\in \partial I(u(t_0))$ and that both $u(t)$ and $I(u(t))$ are differentiable at $t=t_0.$ Then we have
 \[
 I(u(t_0))\le I(v)+(g(t_0), u(t_0)-v)\quad \forall\, v\in H.
 \]
 Take $v=u(t_0\pm \epsilon)$ with $\epsilon \to 0^+$ and use difference quotient  in the usual way  to obtain
 \[
 \frac{d}{dt}I(u(t))\Big |_{t=t_0}=(g(t_0),u'(t_0)).
 \]
 \end{proof}

\subsection{Nonlinear Semigroups  Generated by $\partial I$}

In this subsection, as before, we assume that $I\colon H\to \bar \R $ is a proper, convex  and lower semicontinuous functional.

First we have the following definition.

\begin{defn} Let $D\subseteq H.$ A family $\{S(t)\}_{t\ge 0}$ of mappings from $D$ into $D$ is called a  {\bf $C^0$-semigroup} on $D$
 if  \begin{equation}\label{semi-G}
\begin{cases} S(0)u_0=u_0 \quad \forall\, u_0\in D,\\
S(t+s)u_0 =S(t)S(s)u_0 \quad \forall\, t, s\ge 0,\;\forall\, u_0\in D,\\
\mbox{$u(t)= S(t)u_0\colon [0,\infty)\to H$ is continuous for each $u_0\in D.$}\end{cases}
 \end{equation}
   If, in addition,
\[
\|S(t)u_0-S(t)v_0\|\le \|u_0-v_0\| \quad \forall\, t\ge 0,\; \forall\, u_0,v_0 \in D,
\]
then  $\{S(t)\}_{t\ge 0}$ is called a  {\bf $C^0$-semigroup of contractions} on $D.$
\end{defn}

\begin{thm}\label{G-flow} For each $u_0\in D(\partial I),$ there exists a unique function
\begin{equation}\label{strong}
\mbox{$u\in C([0,\infty);H)$\; with\; $u'\in L^\infty(0,\infty;H)$}
\end{equation}
and $u(t)\in D(\partial I)$ for all $t\ge 0$ that solves the Cauchy problem:
\begin{equation}\label{CP}
\begin{cases}
u'(t)\in -\partial I(u(t)), \quad \mbox{a.e.\;$t\ge 0$},\\
  u(0)=u_0.
\end{cases}
\end{equation}
Let $u(t)=S_0(t)u_0.$ Then $\{S_0(t)\}_{t\ge 0}$ is a $C^0$-semigroup of contractions on $D(\partial I).$
 \end{thm}

 \begin{proof} 1. Let $\l>0.$ Since $A_\l\colon H\to H$ is $\frac{1}{\l}$-Lipschitz continuous,  for each $u_0\in H,$ the ODE
 \begin{equation}\label{app-Yosida}
\begin{cases}
u_\l'(t)+A_\l(u_\l(t))=0 \quad \forall\, t\ge 0,\\
u_\l(0)=u_0
\end{cases}
\end{equation}
 has a unique solution $u_\l\in C^1([0,\infty);H).$ We write this solution as
$
u_\l(t)=S_\l(t)u_0.
$
By uniqueness of solutions, it follows that $\{S_\l(t)\}_{t\ge 0}$ is a $C^0$-semigroup on $H$.
If $v_0\in H$ and $v_\l(t)=S_\l(t)v_0$, then, by the monotonicity of $A_\l$, we have
\[
\frac12 \frac{d}{dt} \|u_\l(t)-v_\l(t)\|^2 = (u_\l'-v_\l', u_\l-v_\l)=-(A_\l(u_\l)-A_\l(v_\l), u_\l-v_\l)\le 0,
\]
and thus, integrating, we have
 \begin{equation}\label{c-tr}
\|S_\l(t)u_0-S_\l(t)v_0\|\le \|u_0-v_0\| \quad \forall\, t\ge 0, \;\forall\, u_0, v_0\in H.
\end{equation}
This shows that $\{S_\l(t)\}_{t\ge 0}$ is a  $C^0$-semigroup of contractions on $H.$  Note that $u_\l(t)$ is exactly the gradient flow for the differentiable functional $I_\l$ on $H$; hence
\[
\frac{d}{dt} I_\l(u_\l(t))= (I_\l'(u_\l),u_\l')=-\|A_\l(u_\l)\|^2.
\]

Our plan is to show that if $u_0\in D(\partial I),$ then, as $\l\to 0^+$,  $S_\l (t)u_0$ converges to a unique function $u(t)=S_0(t)u_0$ satisfying (\ref{CP})  and that $\{S_0(t)\}_{t\ge 0}$ is a  $C^0$-semigroup of contractions on $D(\partial I).$

2. Now assume $u_0\in D(\partial I)$ and $t\ge 0.$ Let $h>0$ and $v_0=u_\l(h)$; then,  $v_\l(t)=S_\l(t)v_0=u_\l(t+h)$ and hence, by (\ref{c-tr}),
\[
\|u_\l(t+h)-u_\l(t)\|\le \|u_\l(h)-u_0\|.
\]
Dividing by $h$ and sending $h\to 0^+$, we have
 \begin{equation}
\|A_\l(u_\l(t))\|=\|u_\l'(t)\|\le \|u_\l'(0)\|=\|A_\l(u_0)\|\le \|A^0(u_0)\|,
\label{gflow-1}\end{equation}
by Theorem \ref{Yosida}(v).

3. Next we take $\l, \mu>0$ and observe
\begin{equation}\label{e-0}
\frac12 \frac{d}{dt} \|u_\l(t)-u_\mu (t)\|^2 = (u_\l'-u_\mu', u_\l-u_\mu) =-(A_\l(u_\l)-A_\mu (u_\mu), u_\l-u_\mu).
\end{equation}
By the definition of $A_\lambda$ and $A_\mu$,
\[
u_\l-u_\mu =\l A_\l(u_\l) +J_\l(u_\l)-J_\mu(u_\mu)-\mu A_\mu(u_\mu).
\]
So we deduce

\[
\begin{split}
(A_\l(u_\l)-A_\mu (u_\mu), u_\l-u_\mu) =  & \, (A_\l(u_\l)-A_\mu (u_\mu), J_\l(u_\l)-J_\mu(u_\mu)) \\
& \,+ (A_\l(u_\l)-A_\mu (u_\mu), \l A_\l(u_\l)-\mu A_\mu(u_\mu))\\
\ge &\, -(\l+\mu)(\|A_\l(u_\l)\|+\|A_\mu(u_\mu)\|)^2 \\
\ge &\, -4(\l+\mu)\|A^0(u_0)\|^2,
\end{split}
\]
where, in the first inequality, we have used $(A_\l(u_\l)-A_\mu (u_\mu), J_\l(u_\l)-J_\mu(u_\mu))\ge 0$ from the monotonicity of $\partial I$ and $A_\l(u_\l)\in\partial I(J_\l(u_\l));$ thus, reflecting this on (\ref{e-0}), we have
\[
\frac{d}{dt} \|u_\l(t)-u_\mu (t)\|^2\le 8(\l+\mu)\|A^0(u_0)\|^2 \quad \forall\, t\ge 0
\]
so that
\begin{equation}
\|u_\l(t)-u_\mu(t)\|^2 \le 8(\l+\mu) t \|A^0(u_0)\|^2 \quad \forall\, t\ge 0.
\label{gflow-2}\end{equation}
This proves that for each $T>0,$ $\{u_\l\}_{\l>0}$ is Cauchy in $C([0,T];H)$ along any sequence $\l\to 0^+.$ Therefore, there exists a function $u\in C([0,\infty);H)$ such that for each $T>0,$
\[
u_\l  \to u \quad\mbox{in $C([0,T];H)$}
\]
as $\l\to 0^+.$  We write
$u(t)=S_0(t)u_0.$  Furthermore,   (\ref{gflow-1}) implies that for each $T>0,$
\[
u_\l'\wcon u' \quad \mbox{in $L^2(0,T; H)$},
\]
and thus
\[
\|u'(t)\|\le \|A^0(u_0)\|, \quad \mbox{a.e.\;$t\ge 0$}.
\]
Hence $u'\in L^\infty(0,\infty;H).$

4. We  prove  (\ref{CP}).    Clearly, $u(0)=u_0.$  To prove the differential inclusion in  (\ref{CP}), we observe that
\[
\|J_\l(u_\l(t))-u_\l(t)\|=\l \|A_\l(u_\l(t))\|=\l\|u_\l'(t)\|\le \l\|A^0(u_0)\|
\]
and hence,  for each $T>0$, $J_\l(u_\l)\to u$ in $C([0,T];H)$ as $\l\to 0^+$. For each $t\ge 0$,
\[
-u_\l'(t)=A_\l(u_\l(t))\in \partial I(J_\l(u_\l(t))).
\]
Thus, given $w\in H$, we have
\[
I(w) \ge I(J_\l(u_\l(t))) -(u_\l'(t), w-J_\l(u_\l(t))).
\]
Consequently, if $0\le s\le t$,
\[
(t-s)I(w) \ge \int_s^t I(J_\l(u_\l(r)))dr -\int_s^t (u_\l'(r), w-J_\l(u_\l(r)))dr.
\]
Sending $\l\to 0^+$, by the lower semicontinuity  of $I$  and Fatou's lemma, which is applicable due to the lower bound estimate (\ref{coer-I}),  we have
\[
(t-s)I(w) \ge \int_s^t I(u(r))dr -\int_s^t (u'(r), w-u(r))dr
\]
for all $0\le s\le t.$ Therefore,
\[
I(w)\ge I(u(t))+(-u'(t), w-u(t))\quad\forall\, w\in H
\]
if $t\in(0,\infty)$ is a Lebesgue point of $u(t)$, $u'(t)$ and $I(u(t)).$ This proves that
\[
u(t)\in D(\partial I)\quad\mbox{and}\quad u'(t)\in -\partial I(u(t)), \quad  \mbox{a.e.\;$t\ge 0$},
\]
which verifies  (\ref{CP}).

5. We now prove that $u(t)\in D(\partial I)$  for all $t >0.$  To see this, fix any $t>0$ and choose a sequence $t_k\to t$ in $(0,\infty)$ such that $u(t_k)\in D(\partial I)$ and $-u'(t_k)\in \partial I(u(t_k))$ for all $k\in\N$.  We may also assume that
\[
u'(t_k)\wcon v \quad \mbox{in $H$.}
\]
 Then
\[
I(w)\ge I(u(t_k)) +(-u'(t_k),w-u(t_k)) \quad \forall\, w\in H.
\]
Letting $k\to \infty$, by the continuity of $u$ and lower semicontinuity  of $I$, we obtain
\[
I(w)\ge I(u(t))+(-v,w-u(t)) \quad \forall\, w\in H,
\]
which proves that $-v\in \partial I(u(t))$ and thus $u(t)\in D(\partial I).$

6. We  prove the uniqueness of function $u.$ Let $\tilde u$  be any function with (\ref{strong}) that solves (\ref{CP}). Then
\[
\frac12 \frac{d}{dt} \|u(t)-\tilde u(t)\|^2 =(u'-\tilde u', u-\tilde u)\le 0, \quad  \mbox{a.e.\;$t\ge 0$}
\]
by the monotonicity of $\partial I$ since $-u'\in \partial I(u)$ and $-\tilde u'\in \partial I(\tilde u).$ Hence $\|u(t)-\tilde u(t)\|^2\le \|u(0)-\tilde u(0)\|^2=0$ for all $t\ge 0,$ which proves that $\tilde u=u.$

7. Finally, since $\{S_\l(t)\}_{t\ge 0}$ is  a  $C^0$-semigroup of contractions on $H$ for all $\l>0,$ it easily follows that $\{S_0(t)\}_{t\ge 0}$ is  a   $C^0$-semigroup of contractions  on $D(\partial I)$.
 This completes the  proof.
 \end{proof}

\begin{lem}\label{G-flow-2}  Given $u_0\in \overline{D(\partial I)}$ and $t\ge0,$
there exists $S(t) u_0 \in \overline{D(\partial I)}$ such that
\[
S_0(t)u_k \to S(t)u_0\quad\mbox{in $H$}
\]
for any sequence $\{u_k\}$ in $D(\partial I)$ with $u_k\to u_0.$
Furthermore, the family $\{S(t)\}_{t\ge 0}$ of mappings forms a  $C^0$-semigroup of contradictions on $\overline{D(\partial I)}.$
\end{lem}

\begin{proof} Let $\{u_k\}$ be any sequence in $D(\partial I)$ with $u_k\to u_0.$ Note that
\begin{equation}\label{c-1}
\|S_0(t)u_k-S_0(t)u_\ell\|\le \|u_k-u_\ell\| \quad \forall\; k,\ell=1,2,\ldots.
\end{equation}
Thus, $\{S_0(t)u_k\}$ is a Cauchy sequence in $D(\partial I)\subseteq H$; hence, the limit
\[
 S(t)u_0:=\lim_{k\to\infty}S_0(t)u_k \in \overline{D(\partial I)}
\]
exists and  is independent of the choice of sequence $\{u_k\}.$  So $\{S(t)\}_{t\ge 0}$ is well-defined on $\overline{D(\partial I)}.$

We  observe that
\[
 S(0)u_0 =\lim_{k\to \infty} S_0(0)u_k =\lim_{k\to \infty} u_k=u_0.
\]
This shows that $(\ref{semi-G})_1$ holds.

Let $t,s\ge 0.$ Then we have
\[
 S(t+s)u_0=\lim_{k\to \infty} S_0(t+s)u_k =\lim_{k\to \infty} S_0(t)S_0(s)u_k \]
 \[
 = S(t) \lim_{k\to \infty}  S_0(s)u_k = S(t) S(s)u_0,
\]
which shows that $(\ref{semi-G})_2$ holds.

Next, taking $\ell\to\infty$ in (\ref{c-1}), we have
\[
\| S_0 (t)u_k-S(t)u_0\|\le \|u_k-u_0\| \quad\forall\, t\ge 0,\;\forall \, k=1,2,\ldots,
\]
from which it follows that
\[
\begin{split}
\|S(t)u_0-S(s)u_0\| \le &\, \|S(t)u_0-S_0(t)u_k\|+\|S_0(t)u_k-S_0(s)u_k\|\\
& \,+\|S_0(s)u_k-S(s)u_0\| \\
\le & \, \|S_0(t)u_k-S_0(s)u_k\|+ 2\|u_k-u_0\|.
\end{split}
\]
Since $\|S_0(t)u_k-S_0(s)u_k\|\to 0$ as $t\to s$ for each $u_k\in D(\partial I),$ it follows that
\[
\lim_{t\to s} \|S(t)u_0-S(s)u_0\|=0,
\]
which proves $(\ref{semi-G})_3.$

Finally, given $v_0\in  \overline{D(\partial I)},$  let $\{v_k\}$ be a sequence in $D(\partial I)$ such that $v_k\to v_0.$ Since $\{S_0(t)\}_{t\ge0}$ is of contractions on $D(\partial I)$, we have
\[
\|S_0(t)u_k-S_0(t)v_k\|\le \|u_k-v_k\|,
\]
and letting $k\to \infty$, we have
\[
\|S(t)u_0-S(t)v_0\|\le \|u_0-v_0\|\quad \forall\, t\ge 0,\;\forall\, u_0,v_0\in  \overline{D(\partial I)},
\]
which confirms the contraction property of $\{S(t)\}_{t\ge0}.$
\end{proof}

The following  theorem  summarizes  the important properties of the semigroup $\{S(t)\}_{t\ge 0}.$

\begin {thm}\label{main-semi-g} Let $u_0\in \overline{D(\partial I)}$ and $u(t)=S(t)u_0.$  Then
 \begin{itemize}
\item[(i)]
$ \;\;  t^{1/2}u'(t)\in L^2(0,T;H) \;\;   \forall\, T>0;$
\item[(ii)]
$ \;\;   I(u)\in L^1(0,T) \;\;  \forall\, T>0;$
\item[(iii)]  $\;\; u(t)\in D(\partial I)\;\mbox{and}\;  -u'(t)  \in \partial I(u(t)), \;\; \mbox{a.e.\;$t>0$};$
 \item[(iv)] $ \;\; -u'(t)=A^0(u(t)),\;\; \mbox{a.e.\;$t>0$};$
\item[(v)] $\;\; u(t)\in D(\partial I) \;\; \forall\,t>0,$ and $I(u(t))$ is nonincreasing on $(0,\infty);$
\item[(vi)] $\;\; u'\in L^2(0,T;H)$ and $I(u)\in W^{1,1}(0,T)$ $\,\forall\,T>0$ if $u_0\in D(I);$
\item[(vii)] $\;\; u(t)=S(t)u_0$ is the unique solution in the class
\begin{equation}\label{CP-b}
\{v\in C^0([0,\infty);H):v'\in L^\infty(\delta,\infty;H)\;\;\forall\,\delta>0\}
\end{equation}
\;to the Cauchy problem of gradient flow:
\begin{equation}\label{CP-a}
 \begin{cases}  u'(t)\in -\partial I (u(t)), \quad \mbox{a.e.\;$t>0$},\\
 u(0)=u_0.
 \end{cases}
\end{equation}
\;\;Moreover, $
\|u'\|_{L^\infty(\delta,\infty;H)}\le\|A^0(u(\delta))\|$ for all $\delta>0.$
 \end{itemize}
\end{thm}

\begin{proof}
1.  Let $u^0_k\in D(\partial I)$ be such that $u^0_k\to u_0$ as $k\to\infty.$ Set $u_k(t)=S_0(t)u^0_k.$  Then $u_k'\in L^\infty(0,\infty;H)$, and
\[
-u_k'(t)\in \partial I(u_k(t)), \quad \mbox{a.e.$\,t>0$}.
\]
By Lemma \ref{chain-rule},
\[
t\|u_k'(t)\|^2 +t (I(u_k(t)))'=0, \quad \mbox{a.e.$\,t>0$},
\]
and thus
\begin{equation}\label{i2}
\int_0^Tt\|u_k'(t)\|^2 \,dt + I(u_k(T)) T =\int_0^T I(u_k(t))\,dt\quad\forall\,T>0.
\end{equation}
Moreover, note that for a.e. $t>0$,
\begin{equation}\label{i-23}
\begin{split}
I(u_k(t)) & \le I(x_0)+(u_k'(t), x_0-u_k(t))\\
& =I(x_0)-\frac12 \frac{d}{dt} \|u_k (t)-x_0\|^2\quad\forall x_0\in D(I).
\end{split}
\end{equation}

2. Fix any $T>0$. Taking $x_0=u_k(T)$ in (\ref{i-23}) and integrating, we have
\[
\int_0^TI(u_k(t))\,dt \le I(u_k(T)) T +\frac12  \|u_k (0)-u_k(T)\|^2,
\]
which, combined with (\ref{i2}), gives
\begin{equation*}
\int_0^Tt\|u_k'(t)\|^2 \,dt \le  \frac12  \|u_k (0)-u_k(T)\|^2.
\end{equation*}
This implies that after passing to a subsequence if necessary, for some $v\in L^2(0, T;H)$,
\[
t^{1/2}u'_k(t) \rightharpoonup v\quad\mbox{in $L^2(0, T;H)$.}
\]
Let $0<\delta<T.$ Then  $\{u_k'\}$ is bounded in $L^2(\delta, T;H).$
Since $u_k\to u$ in $C([0,T];H)$ as $k\to \infty,$ it follows  that after passing again to a subsequence if necessary, $u'_k \rightharpoonup u'$ in $L^2(\delta, T;H)$; thus
\[
t^{1/2}u'_k(t) \rightharpoonup t^{1/2}u'(t)\quad\mbox{in $L^2(\delta, T;H)$}.
\]
Hence $t^{1/2}u'(t)=v(t)$ for a.e. $t\in(0,T)$, which  confirms (i).

3. Let $x_0\in D(I)\ne\emptyset$. Integrating (\ref{i-23}), we have
\[
\int_0^TI(u_k(t))\,dt \le I(x_0) T +\frac12  \|u_k (0)-x_0\|^2-\frac12  \|u_k (T)-x_0\|^2.
\]
Letting $k\to \infty,$ by the lower semicontinuity  of $I$ and Fatou's lemma, in view of  (\ref{coer-I}),   we have
\begin{equation}\label{i3}
\int_0^TI(u(t))\,dt \le I(x_0) T +\frac12  \|u_0-x_0\|^2-\frac12  \|u (T)-x_0\|^2<\infty.
\end{equation}
On the other hand, by (\ref{coer-I}),
\[
\int_0^T I(u(t))\,dt\ge (-C-C\|u\|_{C([0,T];H)})T> -\infty,
\]
which confirms (ii).

4. Let $x\in D(\partial I)$ and $y\in\partial I(x).$ Then, by monotonicity,
\[
 \frac12\frac{d}{dt} \|u_k(t)-x\|^2 =(u_k'(t),u_k(t)-x) \le  (-y, u_k(t)-x), \quad \mbox{a.e.\;$t>0$}.
\]
Let $t>s\ge0.$ Integrating, we have
\[
 \frac12 \|u_k(t)-x\|^2- \frac12 \|u_k(s)-x\|^2   \le  \int_s^t (-y, u_k(\tau)-x)\,d\tau.
 \]
Letting  $k\to\infty$, we have
\[
 \frac12 \|u(t)-x\|^2- \frac12 \|u(s)-x\|^2   \le  \int_s^t (-y, u(\tau)-x)\,d\tau,
 \]
and,  rearranging and dividing by $t-s>0,$  we have
 \[
 \left ( \frac{u(t)-u(s)}{t-s}, \frac{u(t)+u(s)}{2}-x\right ) \le \frac{1}{t-s} \int_s^t (-y, u(\tau)-x)\,d\tau.
 \]
 Letting $s\to t^-,$ we obtain that for a.e.\,$t>0$,
 \[
  (u'(t), u(t)-x) \le (-y, u(t)-x) \quad \forall \, x\in D(\partial I),\,\forall\, y\in\partial I(x),
  \]
 which, by Lemma \ref{maximal},  implies that
  \[
  u(t)\in D(\partial I)\;\;\mbox{and}\;\; -u'(t)\in \partial I(u(t)),\quad \mbox{a.e.\;$t>0$},
  \]
 confirming (iii).

5. Let  $\delta>0$ be such that $u(\delta)\in D(\partial I).$ Consider the two functions
 \[
 \tilde u(t)=S_0(t)u(\delta)\;\;\mbox{and}\;\; \hat u(t)=u(t+\delta) \quad (t\ge 0).
 \]
Note $\hat u\in C^0([0,\infty);H)$ is such that $\hat u'\in L^2([0,T];H)$ for all $T>0.$
Since both $\tilde u$ and $\hat u$ are a  solution to the Cauchy problem (\ref{CP}) with initial datum $u(\delta)\in D(\partial I),$
we may repeat the proof of uniqueness in Theorem \ref{G-flow} to conclude that $\hat{u}(t)=\tilde{u}(t)$ for all $t\ge0$; that is,
  \[
  u(t+\delta)=S_0(t)u(\delta)\quad \forall\, t\ge 0.
  \]
Thus, by Theorem \ref{G-flow}, we have
\begin{itemize}
\item [(a)] $u(t)\in D(\partial I)$ for all $t\ge \delta$;
\item [(b)] $u'\in L^\infty(\delta,\infty;H)$ and $\|u'\|_{L^\infty(\delta,\infty;H)}\le\|A^0(u(\delta))\|$.
\end{itemize}
Also, by Lemma \ref{chain-rule},
  \begin{equation}\label{i3-1}
  \frac{d}{dt} I(u(t))=-\|u'(t)\|^2\le 0, \quad \mbox{a.e. $t\in (\delta, \infty)$.}
  \end{equation}
Since (a), (b) and (\ref{i3-1}) hold for a.e. $\delta>0,$ (v) is confirmed, and (vii) follows from (iii)  and the proof of uniqueness as in that of uniqueness in Theorem \ref{G-flow}.

6. We now prove (iv). Fix any $t>0$. Since $u(t)\in D(\partial I)$, we have $A^0(u(t))\in \partial I(u(t)).$ Note also from (iii) that $-u'(t+s)\in\partial I(u(t+s))$ for a.e. $s>-t.$ By the monotonicity of $\partial I,$ we now have
\[
(A^0(u(t))+u'(t+s),u(t)-u(t+s))\ge 0,\quad\mbox{a.e. $s>-t$.}
\]
Letting $f(s)=\|u(t+s)-u(t)\|^2$ $(s>-t),$ we see that for a.e. $s>-t,$
\[
\begin{split}
\frac{1}{2}f'(s) & =\frac{1}{2}\frac{d}{ds}\|u(t+s)-u(t)\|^2=(u'(t+s) , u(t+s)-u(t))\\
& \le (A^0(u(t)),u(t)-u(t+s)) \le \|A^0(u(t))\|\sqrt{f(s)}.
\end{split}
\]

Now, let $t>0$ be such that $-u'(t)\in\partial I(u(t)).$ If there is a sequence $s_k\downarrow 0$ such that $f(s_k)=0$ for all $k=1,2,\ldots$, then $0=-u'(t) \in\partial I(u(t))$, which gives $\|A^0(u(t))\|=0$ and hence $A^0(u(t))=0=-u'(t)$.
Next, we assume there exists $s_0>0$ such that $f(s)>0$ for all $0<s<s_0.$ Then from the above calculation, we obtain
\[
\frac{d}{ds}\sqrt{f(s)}\le \|A^0(u(t))\|,\quad\mbox{a.e. $s\in(0,s_0)$;}
\]
thus, integrating,
\[
\sqrt{f(s)}\le s\|A^0(u(t))\|\quad\forall\, 0\le s\le s_0.
\]
Dividing by $s\in(0,s_0]$ and sending $s\to 0^+,$ we have $\|u'(t)\|\le \|A^0(u(t))\|;$ thus  $-u'(t)=A^0(u(t))$ by the uniqueness of $A^0(u(t))$ in $\partial I(u(t))$.
Therefore, (iv) follows from (iii).

7. Assume $u_0\in D(I)$ and let $u_k^0=J_{1/k}(u_0)\in D(\partial I)$ $(k\in\N)$. As in Step 1, let $u_k(t)=S_0(t) u_k^0.$ Since
 $u_k'\in L^\infty(0,\infty;H)$ and $\|u_k'(t)\|^2 + (I(u_k(t)))'=0 $ for a.e.\,$ t>0,$  it follows that $I(u_k(t))$ is Lipschitz and nonincreasing on $[0,\infty)$ and, by (\ref{coer-I}),
 \begin{equation}\label{pr-7}
  \int_\delta^T\|u_k'(\tau)\|^2 d\tau =I(u_k(\delta))-I(u_k(T))  \le I(u_k^0)+C+C\|u_k(T)\|
   \end{equation}
 for all $0<\delta<T.$ On the other hand, as $A_\l(u_0)\in \partial I(J_\l(u_0)),$ we have
 \[
 I(u_0)\ge I(J_{\frac1k}(u_0))+ (A_{\frac1k}(u_0), u_0-J_{\frac1k}(u_0))  =I(u_k^0) + \frac{1}{k}\|A_{\frac1k}(u_0)\|^2 \ge I(u_k^0).
 \]
From this and the lower semicontinuity  of $I$, we have $I(u_k^0)\to I(u_0)$ as $k\to\infty$.  So letting $k\to\infty$ in (\ref{pr-7}), we have
\[
 \int_\delta^T\|u'(\tau)\|^2 d\tau \le I(u_0) +C+C\|u(T)\|<\infty
 \]
for all $0<\delta<T.$ This proves that $u'\in L^2(0,T;H).$ Hence (vi) follows from  (iii) and
Lemma \ref{chain-rule}.
\end{proof}

%

\subsection{Asymptotic Convergence}
In this subsection, we assume that $I:H\to\bar\R$ is a proper, convex and lower semicontinuous functional as before and that $u_0\in\overline{D(\partial I)}$.
We follow  \cite{Bru} to study asymptotic behaviors of the solution $u(t)=S(t)u_0$  as $t\to \infty.$

Throughout this subsection, we assume  that $I$ has  a minimizer on $H.$
Let $(\partial I)^{-1}(0)$ denote the set of all $x\in H$ with $0\in\partial I(x);$ that is, $(\partial I)^{-1}(0)\ne \emptyset$ is the set of minimizers of $I$ on $H$.

Let $y\in (\partial I)^{-1}(0)$ be fixed. Then
 \[
h(t):=-\frac12 \frac{d}{dt} \|u(t)-y\|^2=-(u'(t),u(t)-y) \ge 0, \quad \mbox{a.e.\;$t>0$,}
\]
and hence $h\in L^1(0,\infty).$
By a {\em $(*)$-sequence} we mean a strictly increasing  sequence $\{t_k\}$ in $[1,\infty)$ with $t_k\to \infty$  such that
\[
-u'(t_k)\in\partial I(u(t_k))\quad \forall\, k\in\N\;\;\mbox{and}\;\;h(t_k)\to 0.
\]
 Note that a subsequence of a $(*)$-sequence is also a $(*)$-sequence.

\begin{lem}\label{bruck-lem} There exists an element $u^*\in (\partial I)^{-1}(0)$ such that $u(t_k)\wcon u^*$ for every $(*)$-sequence $\{t_k\}.$
\end{lem}
 \begin{proof}  Since $\|u(t)-y\|$ is nonincreasing in $t\ge 0$, it follows that $\{u(t)\}_{t\ge 0}$ is bounded, and thus every sequence $\{u(s_k)\}$ with $s_k\to \infty$ has a weakly convergent subsequence. Therefore, to show that $\{u(t_k)\}$ converges weakly  to a fixed element $u^*\in (\partial I)^{-1}(0)$ for every $(*)$-sequence $\{t_k\}$, it is enough to show that if
 \begin{equation}
u(t_k)\wcon u_1\;\;\mbox{and}\;\; u(s_k)\wcon u_2
\label{bruck-1}\end{equation}
for some $(*)$-sequences $\{t_k\}$ and $\{s_k\}$, then $u_1=u_2\in (\partial I)^{-1}(0)$.

Note that for all $k=1,2,\ldots,$
\[
I(y)\ge I(u(t_k)) -(u'(t_k), y-u(t_k)) =I(u(t_k)) -h(t_k).
\]
By Mazur's theorem, $I$ is  also sequentially weakly lower semicontinuous on $H;$ hence,  (\ref{bruck-1}) implies  $I(y)\ge I(u_1)$; so $I(u_1)=I(y)$ and $u_1\in (\partial I)^{-1}(0).$ Similarly, $u_2\in (\partial I)^{-1}(0).$ For $j=1,2$, since    $\|u(t)-u_j\|$ is nonincreasing in $t\ge 0$ and thus converges as $t\to\infty,$ we have
 \begin{equation}\label{bruck-2}
\begin{split} \lim_{k\to \infty} \|u(s_k)-u_1\|=\lim_{k\to \infty} \|u(t_k)-u_1\|,\\
 \lim_{k\to \infty}\|u(t_k)-u_2\|=\lim_{k\to \infty} \|u(s_k)-u_2\|.\end{split}
\end{equation}
Note that
 \begin{equation}
\label{bruck-3}
\|u(t)-u_2\|^2=\|u(t)-u_1\|^2 + 2 (u(t)-u_1,u_1-u_2) +\|u_2-u_1\|^2.
\end{equation}
Sending $k\to \infty$ in (\ref{bruck-3}) with $t=t_k$ and $s_k$,  (\ref{bruck-2}) implies
$
\|u_1-u_2\|^2=-\|u_1-u_2\|^2
$
and thus $u_1=u_2\in (\partial I)^{-1}(0)$  as required.
 \end{proof}

 \begin{thm} \label{asym-thm}
$u(t)\wcon u^*$ as $t\to \infty$ for some   minimizer $u^*$ of $I$ on $H$.
 \end{thm}

  \begin{proof} 1. Let $u^*\in (\partial I)^{-1}(0)$ be the element determined in Lemma \ref{bruck-lem}. We  show $u(t)\wcon u^*$  as $t\to \infty.$ To this end, we call a strictly increasing sequence $\{s_k\}$ in $[1,\infty)$ an {\em almost-$(*)$-sequence} if
  \[
  \mbox{there exists a $(*)$-sequence $\{t_k\}$ with $s_k-t_k\to 0.$}
  \]
By Theorem \ref{main-semi-g}(vii), we have
$
\|u(t)-u(s)\|\le \|A^0(u(1))\| |t-s|$ for all $ t,s\ge 1.$ Hence,
Lemma \ref{bruck-lem} implies  that $u(s_k)\wcon u^*$ for every almost-$(*)$-sequence $\{s_k\}.$

2. Now let $\{s_k\}$ be any strictly increasing sequence in $[1,\infty)$ with $s_k\to \infty$. We claim that $\{s_k\}$ has an almost-$(*)$-subsequence. Given $d>0$, put
\[
P_d=\{t\in [1,\infty): \mbox{$-u'(t)\in\partial I(u(t))$  and $h(t)<d$}\}.
\]
Since $h\in L^1(0,\infty)$, the set $[1,\infty)\setminus P_d$ has finite measure; hence only finite number of intervals $(s_k-d,s_k+d)$ can fail to intersect $P_d.$  That is, for each $d>0$, there exists an integer $m=m(d)\ge1$ such that, for each $k\ge m$, $|t-s_k|<d$ for some $t=t(k)\in P_d$.  Then the existence of an almost-$(*)$-subsequence of $\{s_k\}$ is obvious.

3. Finally, suppose $\{u(t)\}$ does not converge weakly to $u^*$ as $t\to \infty.$ Since $\{u(t)\}_{t\ge 0}$ is bounded, there must be a sequence $s_k\uparrow \infty$ such that $u(s_k)\wcon \tilde u$ for some $\tilde u\ne u^*.$ However, $\{s_k\}$ has an almost-$(*)$-subsequence $\{s_{k_j}\}$. But then Step 1 shows that $u(s_{k_j})\wcon u^*$, a contradiction as $u^*\ne \tilde u.$
 \end{proof}

 \section{Variational and Semigroup Solutions Coincide: \\ Proof of Theorems \ref{main-1} and \ref{main-2}\label{s-3}}

In what follows,  let $\F$ be  given as in (\ref{abs-F}), where we assume $1<p<\infty$.  Denote $H=L^2(\Omega;\R^N)$ and $(u,v)=(u,v)_{L^2(\Omega)}$ and $\|u\|=\|u\|_{L^2(\Omega)}$ for $u,v\in H.$  We define functional  $I\colon  H\to\bar \R$ by
 \begin{equation}\label{fun-I}
 I(u)=\begin{cases} \F(u)& \mbox{if $u\in H\cap W_0^{1,p}(\Omega;\R^N)$,}\\[1ex]
 \infty & \mbox{if $u\in H\setminus W_0^{1,p}(\Omega;\R^N).$}
 \end{cases}
\end{equation}

It is easily seen that $D(I)=D(\F)$,
\begin{equation}\label{fun-I-1}
\partial I(u)=\mathcal F(u)\;\;\forall\,u\in D(\F)\;\;\mbox{and}\;\;D(\partial I)=D(\mathcal F),
\end{equation}
and thus the Cauchy problems (\ref{CP-0}) and (\ref{CP-a}) are equivalent.

\begin{lem}\label{lem-1} Assume that $\F$ is coercive and lower semicontinuous on $W_0^{1,p}(\Omega;\R^N)$  and convex on $L^2(\Omega;\R^N)\cap W_0^{1,p}(\Omega;\R^N).$
 Then  $I$ is proper, convex and lower semicontinuous on $L^2(\Omega;\R^N);$ in particular, $\overline{D(\partial I)}= \overline{D(\F)}.$
 Moreover, if in addition $\F$ satisfies (\ref{ass-0}), then $\overline{D(\partial I)}=\overline{D(\F)}=L^2(\Omega;\R^N).$
\end{lem}

\begin{proof} Clearly,  $I$ is proper as  $I(0)=\F(0)<\infty.$

To show $I$ is convex on $L^2(\Omega;\R^N),$ let $u,v\in L^2(\Omega;\R^N)$ and $\lambda\in (0,1).$ If at least one of $u$ and $v$ is not in $W_0^{1,p}(\Omega;\R^N),$ then the inequality
\begin{equation}\label{cx-a}
I(\lambda u+(1-\lambda)v)\le \lambda I(u)+(1-\lambda)I(v)
\end{equation}
holds automatically. If both $u,v\in W_0^{1,p}(\Omega;\R^N),$ then $\lambda u+(1-\lambda)v\in W_0^{1,p}(\Omega;\R^N)$; thus inequality (\ref{cx-a}) follows from the convexity of $\F$ on $L^2(\Omega;\R^N)\cap W_0^{1,p}(\Omega;\R^N).$ This proves the convexity of $I$ on $L^2(\Omega;\R^N).$

To prove the lower semicontinuity of $I$ on $L^2(\Omega;\R^N)$, it suffices to show that
\[
\mbox{$u_j\to u$ in $L^2(\Omega;\R^N)$,}  \;\;   \bar I:=\liminf_{j\to\infty} I(u_j) <\infty \;\;  \imply \;\;  I(u)\le \bar I.
\]
Choose a subsequence $\{u_{j_k}\}$ of $\{u_j\}$ such that $\lim\limits_{k\to\infty} I(u_{j_k})=\bar I$. Let $\epsilon>0.$ Then there exists an integer $K\ge1$ such that
\begin{equation}\label{cx-1}
I(u_{j_k})<\bar I +\epsilon\quad \forall\, k\ge K.
\end{equation}
This implies that $u_{j_k}\in W_0^{1,p}(\Omega;\R^N)$ for all $k\ge K.$ Hence $\F(u_{j_k})=I(u_{j_k})<\bar I +\epsilon$ $(k\ge K);$ thus, from the coercivity of $\F$ on $W_0^{1,p}(\Omega;\R^N)$,  it follows that  $\{u_{j_k}\}_{k=K}^\infty$ is bounded in $W_0^{1,p}(\Omega;\R^N)$ after passing to a subsequence if necessary.  Upon taking a further subsequence that we do not relabel, we have $u_{j_k}\wcon \tilde u$ in $W_0^{1,p}(\Omega;\R^N)$ for some $\tilde u\in W_0^{1,p}(\Omega;\R^N)$. However,  since $u_{j_k}\to u$ in $L^2(\Omega;\R^N)$, it follows that $u=\tilde u\in L^2(\Omega;\R^N)\cap W_0^{1,p}(\Omega;\R^N).$ Hence $u_{j_k}\wcon u$ in $W_0^{1,p}(\Omega;\R^N).$ Since $1<p<\infty,$ by Mazur's lemma, there exist a function $\Lambda:\N\to\N$ with $\Lambda(i)\ge i$ $(i\in\N)$ and a sequence of finite sets of real numbers $\{\lambda(i)_k : i\le  k\le \Lambda(i)\}$ $(i\in\N)$ with $\lambda(i)_k\ge0$ for $k=i,\ldots,\Lambda(i)$ and $\sum_{k=i}^{\Lambda(i)}\lambda(i)_k=1$ such that letting $v_i= \sum_{k=i}^{\Lambda(i)}\lambda(i)_k u_{j_k}$ $(i\in\N)$, we have $v_i\to u$ in $W_0^{1,p}(\Omega;\R^N).$
On the other hand, the convexity of $\F$ on $L^2(\Omega;\R^N)\cap W_0^{1,p}(\Omega;\R^N)$ and (\ref{cx-1}) imply that for all $i\ge K,$
\[
\F(v_i)=\F\bigg (\sum_{k=i}^{\Lambda(i)}\lambda(i)_k u_{j_k}\bigg )\le \sum_{k=i}^{\Lambda(i)}\lambda(i)_k \F(u_{j_k})<\bar I +\epsilon;
\]
thus, the lower semicontinuity of $\F$ on $W_0^{1,p}(\Omega;\R^N)$ implies that
\[
\F(u)\le \liminf_{i\to\infty} \F(v_i) \le \bar I +\epsilon.
\]
This proves that $I(u)=\F(u)\le \bar I$ as desired.

Finally, since $C^\infty_0(\Omega;\R^N)$ is dense in $L^2(\Omega;\R^N),$ it follows that  if  $C^\infty_0(\Omega;\R^N)\subset D(\F)$, then by Lemma \ref{pro-dense},
 \[
 \overline{D(\partial I)}=\overline{D(I)}=\overline{D(\F)}=L^2(\Omega;\R^N).
 \]
The proof is now complete.
\end{proof}

\subsection*{Proof of Theorem \ref{main-1}}  By Lemma \ref{lem-1}, the functional $I$ is proper, convex and lower semicontinuous on $H=L^2(\Omega;\R^N)$. Consequently, Theorem \ref{main-1} follows from (\ref{fun-I-1}) and Theorems \ref{G-flow} and \ref{main-semi-g}.

\begin{lem}\label{lem-11} Assume that $\F$ is coercive and lower semicontinuous on $W_0^{1,p}(\Omega;\R^N)$  and convex on $L^2(\Omega;\R^N)\cap W_0^{1,p}(\Omega;\R^N).$ Let $u_0\in \overline{D(I)}=\overline{D(\F)}$ and $u(t)=S(t)u_0.$ Then $u$ satisfies the variational inequality (\ref{var-sol}) for all $T>0.$
\end{lem}
\begin{proof}
Let $T>0$ and
$
  v\in  L^p(0,T; W_0^{1,p}(\Omega;\R^N))$ with $\partial_t v\in L^2(\Omega_T;\R^N)$ and $v(0) \in L^2(\Omega;\R^N).$
 Then  $v\in C^0([0,T];H)$ and $v'\in L^2(0,T;H).$ By Theorem \ref{main-semi-g}, we have $ u\in W^{1,\infty}(\delta,T;H)$ for all $0<\delta<T$, $-u'(t)\in \partial I(u(t))$ for a.e. $t>0$,  and $u(t)\in D(\partial I)$ for all $t>0.$ Thus,
\begin{equation}\label{v-ineq}
\begin{split}
\F(u(t))  =I(u(t)) & \le I(v(t)) +(u'(t), v(t)-u(t))\\
& = \F(v(t)) +(u'(t), v(t)-u(t)), \quad \mbox{a.e. $t>0$}.
\end{split}
\end{equation}
Since $\F(u)=I(u)\in L^1(0,T)$, we can integrate (\ref{v-ineq}) over $[\delta,\tau]\subset(0,T]$ to obtain
\[
\begin{split} \int_\delta^\tau \F(u)\,dt &\le \int_\delta^\tau \F(v)\,dt +\int_\delta^\tau (u', v-u)\,dt
\\
&= \int_\delta^\tau \F(v)\,dt+ \int_\delta^\tau (v',v-u)\,dt -\int_\delta^\tau (v'-u',v-u)\,dt \\
&=  \int_\delta^\tau\F(v)\,dt + \int_\delta^\tau (v',v-u)\,dt -\int_\delta^\tau \frac12  \frac{d}{dt}\|v(t)-u(t)\|^2\,dt \\
&=  \int_\delta^\tau\F(v)\,dt + \int_\delta^\tau (v',v-u)\,dt -\frac12 \|v(\tau)-u(\tau)\|^2 +\frac12 \|v(\delta)-u(\delta)\|^2.
\end{split}
\]
Letting $\delta\to 0^+$,   (\ref{var-sol}) follows as  $ u,v\in C^0([0,T];H).$
\end{proof}

\begin{pro}\label{pf-main-2} Assume that $\F$ satisfies  condition (\ref{strong-coer}) and  is  lower semicontinuous on $W_0^{1,p}(\Omega;\R^N)$ and convex on $L^2(\Omega;\R^N)\cap W_0^{1,p}(\Omega;\R^N).$  Let  $u_0\in \overline{D(I)}=\overline{D(\F)}$ and $T>0.$  Then $u$ is a  variational solution of the gradient flow  associated to  the functional $\F$ with initial datum $u_0$ in the sense of Definition \ref{v-sol} if and only if $u(t)=S(t)u_0$ for all $t\in [0,T].$
 \end{pro}

\begin{proof}
First, suppose that $u\in C^0([0,T];H)\cap L^p(0,T; W_0^{1,p}(\Omega;\R^N))$ is a  variational solution associated to  the functional $\F$ with initial datum $u_0$ in the sense of Definition \ref{v-sol}. By (\ref{strong-coer}) and (\ref{var-sol}) with the test function $v=0$, we see that
\[
\begin{split}
-\infty<\int_0^T (\nu\|Du(t)\|_{L^p(\Omega)}^p-L)\,dt & \le \int_0^T \F(u)\,dt \\
& \le \F(0)T +\frac{1}{2}\|u_0\|^2-\frac{1}{2}\|u(T)\|^2<\infty;
\end{split}
\]
that is, $\F(u)\in L^1(0,T)$ and thus $u(t)\in D(\F)$ for a.e.\,$t\in (0,T).$ Next, choose a sequence $\{u_k^0\}$ in $D(\partial I)$ so that $u_k^0\to u_0$ in $H$ as $k\to \infty.$ Let $u_k(t)=S_0(t)u_k^0$ for $0\le t\le T$. It follows  from Theorem  \ref{G-flow}, Theorem \ref{main-semi-g}(v)(vi),  and (\ref{strong-coer}) that $u_k \in  C^0([0,T];H)$, $u_k'\in L^\infty (0,T;H)$, and for all $0\le t\le T$,
 \[
 \nu  \|Du_k (t)\|^p_{L^p(\Omega)} \le \F(u_k(t)) +L =I(u_k(t)) +L\le I(u_k^0)+L <\infty.
 \]
In particular, $u_k \in  L^\infty(0,T; W^{1,p}_0(\Omega;\R^N)).$ Thus we can take $v=u_k$ as a test function in (\ref{var-sol}) to have
\[
 \begin{split} \int_0^\tau \F(u)\,dt \le &\int_0^\tau \F(u_k)\, dt +\int_0^\tau (u_k'(t), u_k(t)-u(t)) dt \\
 &+ \frac12 \|u_k^0-u_0\|^2 - \frac12 \|u_k(\tau)-u(\tau)\|^2 \quad\forall\,\tau\in(0,T].\end{split}
 \]
 Since $-u_k'(t)\in \partial I(u_k(t))$ and $u(t)\in D(\F)=D(I)$ for a.e.\,$t\in (0,T)$, we have
 \[
 \F(u_k(t))+ (u_k'(t), u_k(t)-u(t))\le \F(u(t)), \quad \mbox{a.e.\,$t\in (0,T).$}
 \]
 Thus the previous variational inequality yields that
 \[
  \|u_k(\tau)-u(\tau)\|\le \|u_k^0-u_0\|\to 0 \quad \mbox{as $k \to \infty$.}
  \]
 By Lemma \ref{G-flow-2}, we have
  $u_k(\tau)=S_0(\tau)u_k^0 \to S(\tau)u_0$ in $H.$ Hence $u(\tau)=S(\tau)u_0$ for all $\tau\in (0,T].$ Clearly they are also equal at $\tau=0$ as  both are continuous at $0.$

Conversely, suppose that $u(t)=S(t)u_0$ for all $t\in[0,T]$. By Lemma \ref{lem-11},  it suffices to verify  that  $u\in L^p(0,T;W^{1,p}_0(\Omega;\R^N)).$  By Theorem \ref{main-semi-g}(v), we have $u(t)\in D(\partial I)\subset D(I)$ for all $t>0;$  hence $I(u(t))=\F(u(t))$ for all $t>0.$ By condition  (\ref{strong-coer}) and Theorem \ref{main-semi-g}(ii), we have
 \[
 \nu \int_0^T \|Du(t)\|^p_{L^p(\Omega)}\,dt\le \int_0^T I(u(t))\,dt + LT<\infty.
 \]
Therefore, $ u\in L^p(0,T;W^{1,p}_0(\Omega;\R^N)).$ This completes the proof.
\end{proof}

\subsection*{Proof  of Theorem \ref{main-2}}
Thanks to Proposition \ref{pf-main-2}, we only need to verify (i)--(iv) of Theorem \ref{main-2}.

\begin{proof}[Proof of (i) and (ii)] Let $u_0\in  \overline{D(\F)}$ and $u(t)=S(t)u_0$.  By  (\ref{strong-coer}), we have
 \[
 \nu \|Du(t)\|^p_{L^p(\Omega)} \le \F(u(t)) +L =I(u(t)) +L\quad\forall\,t>0.
 \]

By Theorem  \ref{main-semi-g}(v), for any $\delta>0$,
\[
\nu \|Du(t)\|^p_{L^p(\Omega)}\le I(u(\delta)) +L\quad\quad\forall\,t>\delta;
\]
thus, $u\in  L^\infty(\delta,\infty; W^{1,p}_0(\Omega;\R^N)).$

If $u_0\in  D(\F)=D(I),$ then, by Theorem  \ref{main-semi-g}(v)(vi),
\[
\nu \|Du(t)\|^p_{L^p(\Omega)}\le I(u_0) +L\quad\quad\forall\,t>0;
\]
that is, $u\in  L^\infty(0,\infty; W^{1,p}_0(\Omega;\R^N)).$
\end{proof}

\begin{proof}[Proof of (iii)]  This is an immediate consequence of Lemma \ref{lem-1}.
\end{proof}

\begin{proof}[Proof of (iv)] Let $p\ge \frac{2n}{n+2}$ and  $X=L^2(\Omega;\R^N)\cap W_0^{1,p}(\Omega;\R^N).$ If $p\ge 2$, then  $W_0^{1,p}(\Omega;\R^N)\subset L^p(\Omega;\R^N)\subseteq L^2(\Omega;\R^N);$ if $p<2,$ then $p^*= \frac{pn}{n-p} \ge 2$  and thus, by the Sobolev embedding theorem,  $W_0^{1,p}(\Omega;\R^N)\subset L^{p^*}(\Omega;\R^N)\subseteq  L^2(\Omega;\R^N)$. Therefore, we have $X=W_0^{1,p}(\Omega;\R^N).$ Hence  $I=\F$ on $X=W_0^{1,p}(\Omega;\R^N).$

We now show that $I$ has a minimizer on $L^2(\Omega;\R^N);$ this follows from the standard direct method in the calculus of variations. For example,  let $\{u_j\}_{j=1}^\infty$ be a minimizing sequence of $I$ on $L^2(\Omega;\R^N);$ that is,
\[
\lim_{j\to \infty} I(u_j) =\inf_{u\in L^2(\Omega;\R^N)} I(u)=:m \le I(0)=\F(0)<\infty.
\]
Then we can assume $I(u_j)<\infty$ for all $j\ge1$ so that $u_j\in X$ and $I(u_j)=\F(u_j)$ for all $j\ge1$. Coercivity (\ref{strong-coer}) implies that  $\{u_j\}$ is bounded in $W^{1,p}_0(\Omega;\R^N).$ Without relabeling, we assume $u_j\wcon \bar u\in W^{1,p}_0(\Omega;\R^N)$. Given any $\epsilon>0,$   let $K\ge1$ be an integer such that
\[
I(u_j)=\F(u_j)\le m+\epsilon \quad \forall\, j\ge K.
\]
Again, by Mazur's lemma, a sequence $\{v_k\}_{k=1}^\infty$ of convex combinations of $\{u_j\}_{j=K}^\infty$ will strongly converge to $\bar u$ in $W^{1,p}_0(\Omega;\R^N)$. By the convexity of $\F$ on $X$, we
have $\F(v_k)\le m+\epsilon$. By the lower semicontinuity  of $\F$ on $W^{1,p}_0(\Omega;\R^N),$ we have  $I(\bar u)=\F(\bar u)\le m+\epsilon.$ This proves that $\bar u$ is  a minimizer of $I$ on $L^2(\Omega;\R^N).$ Clearly, any minimizer of $I$ on $L^2(\Omega;\R^N)$ is also a minimizer of $\F$ on $W^{1,p}_0(\Omega;\R^N).$

By Theorem \ref{asym-thm}, $u(t)\wcon u^*$ in $L^2(\Omega;\R^N)$ as $t\to \infty,$ for some minimizer  $u^*$ of $I$ on $L^2(\Omega;\R^N),$ which is also a minimizer of $\F$ on $W^{1,p}_0(\Omega;\R^N).$
Since $I(u(t))=\F(u(t))$ is nonincreasing on $t\in (0,\infty),$ it follows that $\lim\limits_{t\to\infty} \F(u(t))$ exists; moreover, $\F(u(t))\le \F(u(1))$ for all $t\ge 1,$ which, by (\ref{strong-coer}), implies that $\{u(t)\}_{t\ge 1}$ is bounded in $W^{1,p}_0(\Omega;\R^N).$ From this and the convergence $u(t)\wcon u^*$ in $L^2(\Omega;\R^N)$ as $t\to \infty,$  it follows  that $u(t)\wcon u^*$ in $W^{1,p}_0(\Omega;\R^N)$ as $t\to \infty,$ which also implies that $u(t)\to u^*$ in $L^p(\Omega;\R^N)$ as $t\to \infty.$

Finally, assume  $p>\frac{2n}{n+2}.$ If $p\ge 2$, then   $u(t)\to u^*$ in $L^p(\Omega;\R^N)$ and thus in $L^2(\Omega;\R^N)$ as $t\to \infty.$ If $\frac{2n}{n+2}<p<2,$ then $p^*=\frac{np}{n-p}>2;$ hence, by the Sobolev compact imbedding theorem,  we still have $u(t)\to u^*$ in $L^2(\Omega;\R^N)$ as $t\to \infty.$

Moreover, from  $-u'(t)\in \partial I(u(t))=\mathcal F(u(t))$ and $\F(u^*)\ge \F(u(t))-(u'(t), u^*-u(t))$ for a.e.\,$t>0$ and $\|u'\|_{L^\infty(1,\infty;H)}\le \|A^0(u(1))\|$, it follows  that \[
\F(u^*)\ge   \F(u(t))-\|A^0(u(1))\| \|u^*-u(t)\|, \quad \mbox{a.e.$\; t\ge 1.$}
\]
Since $ \lim\limits_{t\to\infty}  \|u^*-u(t)\| = 0$ and $\lim\limits_{t\to\infty} \F(u(t))$ exists,  the above inequality shows that
\[
\F(u^*)\ge \lim_{t\to\infty} \F(u(t)),
\]
 which proves that  $ \lim\limits_{t\to\infty}  \F(u(t))=\F(u^*)$ since  $\F(u(t)) \ge \F(u^*)$ for all $t>0.$
\end{proof}

 \section{Integral Convexity and  Monotonicity Conditions}\label{s-4}

In this final section, we make some remarks on $W^{1,p}_0$-integral convexity for functions of the form   $f=f(x,\xi)$, where $f(x,\xi)$ is  $C^1$ in $\xi\in \R^{N\times n}$ for a.e.\,$x\in\Omega$ and measurable in $x\in \Omega$ for all $\xi\in\R^{N\times n}$.

In what follows, let $f(x,\xi)$ be  such a function and  define
\[
 A_f(x,\xi)=D_\xi f(x,\xi)\colon\Omega\times \R^{N\times n}\to \R^{N\times n}.
\]
Let $1\le p\le \infty.$ We shall always  assume  that the  growth condition:
\begin{equation}\label{growth-1}
\begin{cases}  |A_f(x,\xi)|  \le c_0(x)+c_1  |\xi|^{p-1}   &\mbox{if $p\in [1,\infty);$}\\
|A_f(x,\xi)|  \le c_2(x)+c_3 (\xi)   &\mbox{if $p=\infty$,}\end{cases}
  \end{equation}
 holds for  a.e.$\,x\in\Omega$ and all $\xi \in  \R^{N\times n},$
where $c_1>0$ is a constant,  and $c_0\in L^{\frac{p}{p-1}}(\Omega), \, c_2 \in L^{1}(\Omega)$ and $ c_3\in C(\R^{N\times n})$ are some positive functions.

 \begin{pro}\label{eq-condition} Let $1\le p\le \infty.$  Then  $f$ is $W_0^{1,p}$-integral convex  if and only if   $A_f$ is {\em $W_0^{1,p}$-integral monotone} in the sense that
\begin{equation}\label{equiv-1}
\int_\Omega [A_f(x,D\psi+ D\phi)-A_f(x,D\psi)] :D\phi\,dx  \ge 0 \quad \forall\, \phi,\psi \in W_0^{1,p}(\Omega;\R^N).
\end{equation}
If $p=\infty$, then condition (\ref{equiv-1}) is equivalent to
 \begin{equation}\label{equiv-2}
\int_\Omega [A_f(x,D\psi+ D\phi)-A_f(x,D\psi)] :D\phi\,dx  \ge 0 \quad \forall\, \phi,\psi \in C^\infty_0(\Omega;\R^N).
\end{equation}
 Moreover,  if $f(x,\xi)$ is $C^2$ in $\xi,$       then condition (\ref{equiv-2}) is equivalent to
 \begin{equation}\label{equiv-3}
 \int_\Omega D^2f(x,D\psi)(D\phi,D\phi)\,dx\ge 0 \quad \forall\, \phi, \psi \in C_0^\infty(\Omega;\R^N),
  \end{equation}
  where for a.e.\,$x\in\Omega$ and all $\xi,  \eta\in \R^{N\times n}$,  $D^2f(x,\xi)(\eta,\eta)$ is defined by
\[
  D^2f(x,\xi)(\eta,\eta):=  \sum_{i,k=1}^N\sum_{j,l=1}^n \frac{\partial^2 f(x, \xi)}{\partial \xi_{ij}\partial \xi_{kl}}\eta_{ij}\eta_{kl}.
\]
  \end{pro}

\begin{proof} 1. Let $\F(u)=\int_\Omega f(x, Du)\,dx$ for $u\in W_0^{1,p}(\Omega;\R^N)$. Then the  $W_0^{1,p}$-integral convexity of $f$ is equivalent to the convexity of function $h(t)=\F(u+tv)$ in $t\in\R$  for all   $u,v\in W_0^{1,p}(\Omega;\R^N).$ Note that
 condition (\ref{growth-1}) implies that for all   $u,v\in W_0^{1,p}(\Omega;\R^N),$ the function $h(t)=\F(u+tv)$ is differentiable  in $t\in\R,$ with
 \[
h'(t)= \int_\Omega A_f(x,Du+tDv):Dv\,dx,
\]
and thus the convexity of function $h(t)=\F(u+tv)$ in $t\in\R$   is equivalent to that  $h'(t)$ is  nondecreasing  in $t\in\R.$ Note that for  $t,s\in\R$,
\[
\begin{split}
[h'(t)-h'(s)](t-s) & =\int_\Omega [ A_f(x,Du+tDv)-A_f(x, Du+sDv)]: (t-s)Dv\,dx\\
& = \int_\Omega [A_f(x,D\psi+ D\phi)-A_f(x,D\psi)] :D\phi\,dx,
\end{split}
\]
where  $\psi=u+sv$ and $\phi=(t-s)v.$ Thus, $h'$ is nondecreasing on $\R$ for all   $u,v\in W_0^{1,p}(\Omega;\R^N)$ if and only if (\ref{equiv-1}) holds.

2. If $p=\infty,$ it is easily seen that (\ref{equiv-1}) and (\ref{equiv-2}) are equivalent. Now assume $f$ is $C^2$ in $\xi.$ Let $u,v\in C_0^\infty(\Omega;\R^N)$ and $h(t)=\F(u+tv).$ Then
\[
\begin{split}
h''(t) & = \sum_{i,k=1}^N\sum_{j,l=1}^n \int_\Omega \frac{\partial^2 f}{\partial\xi_{ij}\partial \xi_{kl}}(x, Du+tDv)\frac{\partial v^i}{\partial x_j} \frac{\partial v^k}{\partial x_l}\,dx \\
& =\sum_{i,k=1}^N\sum_{j,l=1}^n \int_\Omega \frac{\partial^2 f}{\partial\xi_{ij}\partial \xi_{kl}}(x, D\psi )\frac{\partial \phi^i}{\partial x_j} \frac{\partial \phi^k}{\partial x_l}\,dx = \int_\Omega D^2f(x,D\psi)(D\phi,D\phi)\,dx,
\end{split}
\]
where  $\psi=u+tv$ and $\phi=v.$ Therefore, $h'$ is nondecreasing on $\R$ for all $u,v\in C_0^\infty(\Omega;\R^N)$ if and only if $h''\ge 0$  on $\R$ for all $u,v\in C_0^\infty(\Omega;\R^N)$, which is the case  if and only if (\ref{equiv-3}) holds.
\end{proof}

\begin{pro} Suppose that $f$ is  integral convex.  Then for almost every $x_0\in\Omega,$ the map  $A_f(x_0,\cdot)\colon \R^{N\times n}\to \R^{N\times n}$ is {\em quasimonotone} in the sense that
\begin{equation}\label{q-mono-03}
\int_G  A_f (x_0,\xi+ D\eta(y)) :D\eta(y)\,dy  \ge 0 \quad \forall\, \xi\in \R^{N\times n}, \, \eta \in C_0^\infty(G;\R^N),
\end{equation}
where $G$ is any nonempty bounded open set in $\R^n.$ Therefore for
almost every $x_0\in\Omega,$ the function
$f(x_0,\cdot)\colon \R^{N\times n}\to\R$ is {\em quasiconvex.}
\end{pro}

\begin{proof}   By a density argument,  it suffices to show that for all $\xi\in \R^{N\times n}$ and $\eta\in C_0^\infty(G;\R^N),$ where  $G\ne\emptyset$ is  bounded and open  in $\R^n$,  one has
\begin{equation}\label{q-mono-04}
\int_G  A_f (x_0,\xi+ D\eta(y)) :D\eta(y)\,dy  \ge 0, \quad  \mbox{a.e.}\; x_0\in \Omega.
\end{equation}

Fix any $\xi\in \R^{N\times n}$ and $\eta\in C_0^\infty(G;\R^N).$ Let  $x_0\in\Omega$,  $0<\delta <\dist(x_0,\partial \Omega)/4$ and $a\in G$. Then there exists $r>0$ such that
\[
G_\epsilon(z):=z+\epsilon (G-a)\subset B_{2\delta}(x_0) \subset \subset  \Omega\quad \forall\,  z\in B_\delta(x_0), \; \epsilon \in (0,r).
\]
For  each $\epsilon \in (0,r),$ we define function $ \phi_{\epsilon}\colon B_\delta(x_0)\times \Omega\to\R^N$ by
\[
 \phi_{\epsilon}(z,x)=\begin{cases} \epsilon \eta(a+\frac{x-z}{\epsilon}) & \mbox{if $x\in G_\epsilon(z),$}\\
0 &\mbox{if $x\in \Omega\setminus G_\epsilon(z).$}\end{cases}
\]
Then for each $z\in B_\delta(x_0),$  we have $\phi_\epsilon(z,\cdot)\in C_0^\infty (\Omega;\R^N)$ with  $\supp\phi_\epsilon(z,\cdot)\subset \bar B_{2\delta}(x_0) \subset \subset  \Omega.$
Also, let  $\psi(x)=\zeta(x) \xi x,$  where $\zeta\in C^\infty_0(\Omega)$ with $\zeta =1$ on $B_{2\delta}(x_0).$
  Thus $\psi\in C_0^\infty(\Omega;\R^N)$ and $D\psi=\xi$ on $B_{2\delta}(x_0).$

  We apply inequality (\ref{equiv-2}) in Proposition \ref{eq-condition} with functions $\psi$ and $\phi=\phi_\epsilon(z,\cdot)$ and integrate it over $z\in B_\delta(x_0)$ to obtain that
\[
\begin{split} 0 & \le \int_{B_\delta(x_0)} \left(\int_\Omega [A_f(x,D\psi(x)+ D_x\phi_\epsilon(z,x))-A_f(x,D\psi(x))] :D_x\phi_\epsilon(z,x)dx\right)dz  \\
&= \int_{B_\delta(x_0)} \left(\int_\Omega [A_f(x,\xi+ D_x\phi_\epsilon(z,x))-A_f(x,\xi)] :D_x\phi_\epsilon(z,x)dx\right )dz \\
&=\int_{B_\delta(x_0)} \left(\int_{G_\epsilon(z)} \Big[A_f \Big(x,\xi+ D\eta \Big(a+\frac{x-z}{\epsilon} \Big)  \Big)-A_f(x,\xi)\Big] :D\eta \Big(a+\frac{x-z}{\epsilon} \Big) \,dx\right)dz
\\
&=\epsilon^n  \int_{B_\delta(x_0)} \left( \int_{G} [A_f (z+\epsilon (y-a),\xi+ D\eta(y) )-A_f (z+\epsilon (y-a),\xi)] :D\eta(y) \,dy\right)dz\\
& =\epsilon^n \int_{G} \left( \int_{B_\delta(x_0)} [A_f  (z+\epsilon (y-a),\xi+ D\eta(y) )-A_f  (z+\epsilon (y-a),\xi)] :D\eta(y) \,dz\right )dy\\
&=\epsilon^n \int_{G} \left( \int_{B_\delta(x_0)+\epsilon(y-a)} [A_f  (z',\xi+ D\eta(y) )-A_f  (z',\xi)] :D\eta(y) \,dz' \right )dy\\
&=\epsilon^n \int_{G} \int_{B_{2\delta}(x_0)} \chi_\epsilon(z',y) [A_f  (z',\xi+ D\eta(y) )-A_f  (z',\xi)] :D\eta(y) \,dz' dy,
\end{split}
\]
where $\chi_\epsilon(z',y)=\chi_{B_\delta(x_0)+\epsilon(y-a)}(z')$ is the characteristic function of the set $B_\delta(x_0)+\epsilon(y-a),$ which lies in $B_{2\delta}(x_0)$ for each $y\in G.$ Consequently, we have
\[
\int_{G} \int_{B_{2\delta}(x_0)} \chi_\epsilon(z',y) [A_f  (z',\xi+ D\eta(y) )-A_f  (z',\xi)] :D\eta(y) \,dz' dy\ge 0\quad \forall\; \epsilon \in (0,r).
\]

Note that $\chi_\epsilon(z',y)\to \chi_{B_\delta(x_0)}(z')$ as $\epsilon\to 0^+$ for  a.e.\,$(z',y)\in B_{2\delta}(x_0)\times G.$
Hence, in the   inequality above, letting $\epsilon\to 0^+$,  by Lebesgue's dominated convergence theorem, we obtain
\[
\begin{split}
0 &\le \int_{G} \int_{B_{\delta}(x_0)}  [A_f  (z',\xi+ D\eta(y) )-A_f  (z',\xi)] :D\eta(y) \,dz' dy\\
&=\int_{B_{\delta}(x_0)} \left (\int_{G}  [A_f  (z',\xi+ D\eta(y) )-A_f  (z',\xi)] :D\eta(y)   dy \right ) dz'. \end{split}
\]
This inequality holds for all $x_0\in\Omega$ and  $0<\delta <\dist(x_0,\partial \Omega)/4;$  therefore, by Lebesgue's differentiation theorem, we have
\[
\int_{G}  [A_f  (x_0,\xi+ D\eta(y) )-A_f  (x_0,\xi)] :D\eta(y) dy\ge 0, \quad  \mbox{a.e.}\; x_0 \in \Omega,
\]
from which (\ref{q-mono-04}) follows because
$
\int_G A_f  (x_0 ,\xi)  :D\eta(y) \,dy=0.
$

Finally, the quasiconvexity of $f(x_0,\cdot)$  follows from the quasimonotonicity condition
(\ref{q-mono-03}) as explained in the introduction. This completes the proof.
\end{proof}

From the results above, we see that  for functions of the simple form $f=f(\xi) \in C^1(\R^{N\times n}),$  the integral convexity of $f$ is equivalent to  the integral monotonicity of $Df\colon \R^{N\times n} \to \R^{N\times n},$   which in turn implies the quasimonotonicity  of $Df$ and thus the quasiconvexity of $f.$

We now give an example to show that  in general the quasimonotonicity of $Df$  does not imply the  integral convexity of  $f=f(\xi).$

For this purpose, we consider the  function $g$ introduced in (\ref{fun-g}); that is,
\begin{equation}\label{fun-g-1}
g(\xi)= |\xi|^4 + k (\det \xi)^2 \quad \forall\, \xi\in \R^{2\times 2},
\end{equation}
where $k\in \R$ is a constant.  Note that
 \begin{equation}\label{g-12}
 \begin{split}
 Dg(\xi) & =4|\xi|^2\xi + 2k(\det \xi ) \cof \xi,
\\
 D^2g(\xi)(\eta,\eta)&  =8(\xi:\eta)^2 +4|\xi|^2|\eta|^2+2k (\eta:\cof \xi)^2  +4k \det \xi\det\eta,
\end{split}
\end{equation}
where $\cof \xi$ is  the cofactor matrix of $\xi\in \R^{2\times 2}$, which satisfies:
\[
\xi:\cof \xi= 2\det\xi\quad\mbox{and}\quad \det(\xi+\eta)=\det\xi+\det\eta+\eta:\cof\xi
\]
for all $ \xi$ and $\eta$ in  $\R^{2\times 2}.$

We introduce the notations:
\[
\xi^+=\frac12(\xi+\cof \xi)\quad\mbox{and}\quad \xi^-=\frac12(\xi-\cof\xi).
\]
Then the following identities can be easily verified:
\begin{equation}\label{conf-m}
\begin{split}   & |\xi|^2=|\xi^+|^2+|\xi^-|^2,\quad
 \det\xi
 =\frac12 (|\xi^+|^2-|\xi^-|^2), \\
 \xi:\eta&=\xi^+:\eta^++\xi^-:\eta^-, \quad
  \eta:\cof \xi  =\xi^+:\eta^+-\xi^-:\eta^-.
\end{split}
\end{equation}

\begin{pro}\label{counter}  The following statements are true:

\begin{itemize}

 \item[(i)]   $Dg$ is quasimonotone if  $0\le k\le 8.$ In fact, we have
\begin{equation}\label{g-1}
\int_\Omega  Dg(\xi+D\phi): D\phi  \ge \frac{8-k}{8} \int_\Omega |D\phi|^4  \quad \forall\, \xi\in \R^{2\times 2},\, \phi\in W^{1,4}_0(\Omega;\R^2).
\end{equation}
 Hence,   if $0\le k<8,$ then  $Dg$  is {\em strongly  quasimonotone on $W^{1,4}_0(\Omega;\R^2)$} (defined in the obvious way).
\item[(ii)] $g$ is integral convex if and only if $-2\le k\le 4,$  in which case $g$ is in fact convex.
 \end{itemize}

 Therefore  if $4<k\le 8$, then $Dg$ is quasimonotone but   $g$ is not integral convex.
\end{pro}

\begin{proof}  1. Let $k\ge 0$ and denote $A(\xi)=Dg(\xi).$ Then elementary  computations show that
 \begin{equation}\label{eq-A}
 \begin{split} (A(\xi+\eta) &-A(\xi))  : \eta  =  4|\xi|^2|\eta|^2+12(\xi:\eta)|\eta|^2 +8(\xi:\eta)^2+4|\eta|^4
 \\
 &+2k\big [2\det \xi\det \eta +2(\det \eta)^2+(\eta:\cof \xi)^2+3\det \eta (\eta:\cof \xi)\big].
 \end{split}
 \end{equation}

 Following \cite{CZ92, Ha95}, we  have
 \[
 \begin{split} (A(\xi+\eta) & -A(\xi))  : \eta     = 4  \big [ |\xi|^2|  \eta|^2-(\xi:\eta)^2\big ] +12\Big (\xi:\eta +\frac12 |\eta|^2 \Big )^2  + |\eta|^4
 \\
 &\quad  +4k\det \xi\det \eta +2k\big [2(\det \eta)^2+(\eta:\cof \xi)^2+3\det \eta (\eta:\cof \xi)\big]\\
 & =   4\big [|\xi|^2|\eta|^2-(\xi:\eta)^2\big ] +12\Big (\xi:\eta +\frac12 |\eta|^2 \Big )^2 +  |\eta|^4
 \\
 &\quad +4k\det \xi\det \eta +2k\bigg [ \Big (\eta:\cof \xi +\frac32  \det \eta\Big )^2  -\frac14(\det \eta)^2\bigg ]\\
 &\ge    |\eta|^4 +4k\det \xi\det \eta -\frac{k}{2} (\det \eta)^2\\
 &\ge    \frac{8-k}{8}   |\eta|^4 +4k\det \xi\det \eta,
  \end{split}
 \]
 where in the first inequality we have dropped three nonnegative terms, and in the second inequality we have used the inequality:
  \begin{equation}\label{det-norm}
 |\eta|^2\ge 2|\det\eta| \quad \forall\, \eta\in \R^{2\times 2}.
 \end{equation}

 Consequently, for all $ \xi\in \R^{2\times 2}$ and $\phi\in W^{1,4}_0(\Omega;\R^2),$ we have
 \[
\begin{split} \int_\Omega  A(\xi+D\phi): D\phi \,dx &=\int_\Omega  [A(\xi+D\phi)-A(\xi)]: D\phi \,dx \\
&\ge \frac{8-k}{8} \int_\Omega |D\phi|^4\,dx+4k \int_\Omega \det\xi \det D\phi\,dx \\
&=\frac{8-k}{8} \int_\Omega |D\phi|^4\,dx,\end{split}
\]
which confirms (\ref{g-1}). Therefore, $Dg$ is quasimonotone if  $0\le k\le 8$ and  strongly quasimonotone on $W^{1,4}_0(\Omega;\R^2)$ if $0\le k<8.$

2.  We show that $g$ is convex for all $-2\le k\le 4.$ Since every $g$ with $k\in[-2,4]$  is  a convex combination of the two functions $g$ with $k=-2$ and $k=4$, it suffices to show that $g$ is convex for $k=4$ and $k=-2.$

First, assume $k=4$. Then
\[
D^2g(\xi)(\eta,\eta)=8(\xi:\eta)^2 +4|\xi|^2|\eta|^2+8(\eta:\cof \xi)^2  +16\det \xi\det\eta.
\]
By (\ref{det-norm}), we have $|\xi|^2|\eta|^2+4\det \xi\det\eta\ge 0$ and thus
\[
D^2g(\xi)(\eta,\eta)\ge  4|\xi|^2|\eta|^2 +16\det \xi\det\eta \ge 0 \quad \forall\, \xi,\,\eta\in \R^{2\times 2}.
\]
This confirms the convexity of function $g$ when $k=4.$

 Next, let $k=-2.$ The convexity of $g$ in this case  is actually established in \cite{DDGR}. Indeed, in this case, we observe that, by identities (\ref{conf-m}),
 \[
 \begin{split}
D^2g(\xi)(\eta,\eta) & =8(\xi:\eta)^2 +4|\xi|^2|\eta|^2-4 (\eta:\cof \xi)^2  -8 \det \xi\det\eta \\
& =8(a+b)^2 +4(c+d)(e+f)-4(a-b)^2-2(c-d)(e-f)\\
& =4(a+b)^2+16ab +2(ce+df) +6(de+cf),
\end{split}
\]
where $a=\xi^+:\eta^+,\,b=\xi^-:\eta^-,\, c=|\xi^+|^2,\,d=|\xi^-|^2,\, e=|\eta^+|^2,$ and $f=|\eta^-|^2.$ Note that
\[
\begin{split}
de+cf &=|\xi^-|^2|\eta^+|^2+|\xi^+|^2|\eta^-|^2 \\
& \ge 2 |\xi^-||\eta^+||\xi^+||\eta^-| \ge 2 |\xi^+:\eta^+||\xi^-:\eta^-|\ge -2 ab,
\end{split}
\]
and
\[
\begin{split}
ce+df & =|\xi^+|^2|\eta^+|^2+|\xi^-|^2|\eta^-|^2 \\
& \ge 2 |\xi^+||\eta^+||\xi^-||\eta^-| \ge 2 |\xi^+:\eta^+||\xi^-:\eta^-|\ge -2 ab.
\end{split}
\]
Hence $2(ce+df) +6(de+cf)\ge -16ab;$ therefore,
\[
D^2g(\xi)(\eta,\eta) =4(a+b)^2+16ab +2(ce+df) +6(de+cf) \ge 0.
\]
 This confirms the convexity of function $g$ when $k=-2.$

3. Let $k>4;$ we show that $g$ is not integral convex.  Let $B(\xi,\eta) = D^2g(\xi)(\eta,\eta),$  as given by the formula  $(\ref{g-12})_2$.
 Given $u=(u^1,u^2)\in W^{1,4}_0(\Omega;\R^2),$  let $\hat u=(u^2,u^1).$ Set
  \[
  \xi=Du=\begin{pmatrix}u^1_{x_1}&u^1_{x_2}\\u^2_{x_1}&u^2_{x_2}\end{pmatrix}\;\;\mbox{and}\;\; \eta=D\hat{u}=\begin{pmatrix}u^2_{x_1}&u^2_{x_2}\\u^1_{x_1}&u^1_{x_2}\end{pmatrix}.
  \]
Then $|\xi|=|\eta|$, $\det\eta=-\det \xi$,
   \[
  \xi:\eta=2 D u^1\cdot D u^2, \;\;  \cof \xi= \begin{pmatrix} u^2_{x_2}&-u^2_{x_1}\\-u^1_{x_2}&u^1_{x_1}\end{pmatrix},\;\;\mbox{and}\;\;  \eta:\cof \xi=0.
  \]
 Therefore,  $ B(Du,D\hat u)=H(Du),$   where $H$ is the function defined by
 \begin{equation}\label{eq-H}
  H(\xi)= 4|\xi|^4 +32 (\xi^1\cdot \xi^2)^2 -4k (\det\xi)^2 \quad \forall\,\xi=\begin{pmatrix}\xi^1\\\xi^2\end{pmatrix}\in \R^{2\times 2}.
  \end{equation}

Let $0<\epsilon_0<k-4$ be fixed. Consider the open set
 \[
 \mathcal U=\{\xi\in \R^{2\times 2} : |\xi|<\sqrt{2}+1, \; H(\xi)<-4\epsilon_0\}
 \]
and  the following four matrices:
  \[
 \xi_1 =\begin{pmatrix}1&0\\0&1\end{pmatrix}, \quad\xi_2=\begin{pmatrix}1&0\\0&-1\end{pmatrix},  \quad  \xi_3=-\xi_1,\quad \xi_4=-\xi_2.
 \]

Note that
\[
H(\xi_i)= 16-4k <- 4\epsilon_0\;\;\mbox{and}\;\; |\xi_i|=\sqrt2 \quad \forall\, i=1,2,3,4;
\]
thus, we have  $ \xi_i\in \U$ for $ i=1,2,3,4.
 $
Also, note that
\[
\rank(\xi_1-\xi_2)=\rank(\xi_3-\xi_4)=1\;\;\mbox{and}\;\; \rank\Big(\frac{\xi_1+\xi_2}{2}-\frac{\xi_3+\xi_4}{2}\Big)=1.
\]
This shows that
\begin{equation}\label{rc-hull-1}
0=\frac12\Big (\frac{\xi_1+\xi_2}{2}+\frac{\xi_3+\xi_4}{2}\Big )\in \U^{lc}\subset \U^{rc}.
\end{equation}
Here, we refer to \cite{D, DM1, MSv2} for the definitions and further properties of the {\em lamination convex hull} $K^{lc}$  and {\em rank-one convex hull} $K^{rc}$ of a general  set   $K\subset \R^{N\times n}$.

 Let $\delta>0$ be a number such that $\cup_{i=1}^4  B_\delta(\xi_i) \subset \U.$ Then, from (\ref{rc-hull-1}), by the convex integration lemma \cite[Theorem 3.1]{MSv2} (see also \cite{DM1}),  there exists  a piecewise affine  function $u\in W^{1,\infty}_0(\Omega;\R^2)$ such that
\[
Du(x)\in  \cup_{i=1}^4  B_\delta(\xi_i),  \quad \mbox{a.e.\,$x\in \Omega$}.
\]
Hence $H(Du(x))<- 4{\epsilon_0},$ a.e.$\,x\in \Omega.$ This shows that
\[
\int_\Omega D^2 g(Du)(D\hat u,D\hat u)\,dx= \int_\Omega B(Du,D\hat u)\,dx =\int_\Omega H(Du(x))\,dx\le - 4{\epsilon_0} |\Omega|.
\]

Finally, by the density of $C^\infty_0(\Omega;\R^2)$ in $W^{1,4}_0(\Omega;\R^2)$, there exist $\psi, \phi \in C^\infty_0(\Omega;\R^2)$ such that
\[
\int_\Omega D^2 g(D\psi)(D\phi,D\phi)\,dx\le \int_\Omega D^2 g(Du)(D\hat u,D\hat u)\,dx +\epsilon_0|\Omega| \le -3{\epsilon_0} |\Omega|<0.
\]
Therefore, by Proposition \ref{eq-condition},  $g$ is not integral convex.

4. Let $k<-2;$ we show that $g$ is not integral convex.    Given $u=(u^1,u^2)\in W^{1,4}_0(\Omega;\R^2),$  let $\psi=(u^1,0)$ and $ \phi=(0,u^2).$ Then
  \[
  \xi=D\psi=\begin{pmatrix}u^1_{x_1}&u^1_{x_2}\\0&0\end{pmatrix}\;\;\mbox{and}\;\; \eta=D\phi=\begin{pmatrix}0&0\\
  u^2_{x_1}&u^2_{x_2} \end{pmatrix}.
  \]
Hence $|\xi|=|D u^1|$, $|\eta|=|D u^2|$, $\det\xi=\det \eta=0,$ $\xi:\eta=0,$ and
   \[
 \eta:\cof \xi=u^1_{x_1}u^2_{x_2}-u^1_{x_2}u^2_{x_1}=\det Du.
 \]
 Therefore,
  $B(D\psi,D\phi)=G(Du),$   where $G(\xi)$ is the function defined by
 \begin{equation}\label{eq-G}
 G(\xi)= 4|\xi^1|^2|\xi^2|^2  + 2k (\det\xi)^2 \quad \forall\,\xi=\begin{pmatrix}\xi^1\\\xi^2\end{pmatrix}\in \R^{2\times 2}.
  \end{equation}

Let $0<\epsilon_0<-(k+2)$ be fixed. Consider a new open set
 \[
\U=\{\xi\in \R^{2\times 2} : |\xi|<\sqrt{2}+1, \; G(\xi)<-2\epsilon_0\}
 \]
and  the  same four matrices $\{\xi_1,\xi_2,\xi_3,\xi_4\}$ defined as in Step 3. Since  $
G(\xi_i)= 4+2k <- 2\epsilon_0,$ we have $ \xi_i\in \U,$ for all $ i=1,2,3,4.
 $

 Let $\delta>0$ be a number such that $\cup_{i=1}^4  B_\delta(\xi_i) \subset \U.$ Then, as in Step 3,   there exists  a piecewise affine  function $u=(u^1,u^2)\in W^{1,\infty}_0(\Omega;\R^2)$ such that
\[
Du(x)\in  \cup_{i=1}^4  B_\delta(\xi_i),  \quad \mbox{a.e.\,$x\in \Omega.$}
\]
Hence $G(Du(x))<- 2 {\epsilon_0},$ a.e.$\,x\in \Omega.$ Setting $\psi=(u^1,0)$ and $\phi=(0,u^2)$, we have
\[
\int_\Omega D^2 g(D\psi)(D\phi,D\phi)\,dx=\int_\Omega B(D\psi,D\phi)\,dx=\int_\Omega G(Du(x))\,dx\le - 2{\epsilon_0} |\Omega|,
\]
which, as in Step 3, proves that  $g$ is not integral convex.

 This completes the proof.
 \end{proof}

\subsection*{Acknowledgments} S. Kim was supported by the National Research Foundation of Korea (grant NRF-2022R1F1A1063379, RS-2023-00217116) and Korea Institute for Advanced Study(KIAS) grant funded by the Korea government (MSIP).

\subsubsection*{\emph{\textbf{Ethical Statement.}}} The manuscript has not been submitted to more than one journal for simultaneous consideration.  The manuscript has not been published previously.

\subsubsection*{\emph{\textbf{Conflict of Interest:}}}  The authors declare that they have no conflict of interest.

\end{document}